\documentclass[reqno]{amsart}
\usepackage{amssymb}
\usepackage{color}
\usepackage{xcolor}
\usepackage[allcolors = blue, colorlinks = true]{hyperref}
\usepackage{mathptmx}
\usepackage{comment}
\usepackage{subfigure}
\usepackage[graphicx]{realboxes}
\numberwithin{equation}{section}
\usepackage{mathtools}
\mathtoolsset{showonlyrefs}
\newtheorem{theorem}{Theorem}[section]
\newtheorem{lemma}[theorem]{Lemma}
\newtheorem{corollary}[theorem]{Corollary} 
\newtheorem{proposition}[theorem]{Proposition}

\theoremstyle{definition}
\newtheorem{definition}{Definition}[section]
\newtheorem{example}{Example}[section]

\theoremstyle{remark}
\newtheorem{remark}{Remark}[section]
\newcommand{\R}{\mathbb{R}}
\newcommand{\Rn}{\R^n}
\newcommand{\il}{\Delta_{\infty}}
\newcommand{\ilN}{\Delta_{\infty}^N}
\newcommand{\plN}{\Delta_{p}^N}
\newcommand{\ue}{u^{\epsilon}}
\newcommand{\la}{\langle}
\newcommand{\ra}{\rangle}

\newcommand{\Om}{\Omega}
\newcommand{\ez}{\epsilon}

\newcommand{\loc}{\textnormal{loc}}

\DeclareMathOperator{\diverg}{div\,}
\DeclareMathOperator{\dist}{dist\,}

\newcommand{\abs}[1]{\left|#1\right|}

\newcommand{\norm}[1]{\left|\left|#1\right|\right|}

\newcommand{\vp}{\varphi}

\newcommand{\kom}[1]{}
\renewcommand{\kom}[1]{{\bf [#1]}}
\newcommand{\eps}{{\varepsilon}}

\begin{document}
%--------------------------------------------------------------------
\title[Second order regularity to the general parabolic $p$-Laplace equation]{A systematic approach on the second order regularity of solutions to the general parabolic $p$-Laplace equation}

\author{Yawen Feng}
\address[Yawen Feng]{Department of Mathematics and Statistics, University of
Jyv\"askyl\"a, PO~Box~35, FI-40014 Jyv\"askyl\"a, Finland; and}
\address{School of Mathematical Science, Beihang University, Changping District Shahe Higher Education Park South Third Street No. 9, Beijing 102206, P. R. China}
\email{yawen.y.feng@jyu.fi}

\author{Mikko Parviainen}
\address[Mikko Parviainen]{Department of Mathematics and Statistics, University of
Jyv\"askyl\"a, PO~Box~35, FI-40014 Jyv\"askyl\"a, Finland}
\email{mikko.j.parviainen@jyu.fi}

\author{Saara Sarsa}
\address[Saara Sarsa]{Department of Mathematics and Statistics, University of Helsinki, PO~Box~68, (Pietari Kalmin katu 5), FI-00014 University of Helsinki, Finland}
\email{saara.sarsa@helsinki.fi}

\begin{abstract}
We study a general form of a degenerate or singular parabolic equation 
$$ u_t-|Du|^{\gamma}\big(\Delta u+(p-2)\Delta_\infty^Nu\big)=0 $$
that generalizes both the standard parabolic $p$-Laplace equation and the normalized version that arises from stochastic game theory.
We develop a systematic approach to study second order Sobolev regularity and show that $D^2u$ exists as a function and belongs to $L^2_\loc$ for a certain range of parameters. In this approach proving the estimate boils down to verifying that a certain coefficient matrix is positive definite. As a corollary we obtain, under suitable assumptions, that a viscosity solution has a Sobolev time derivative belonging to $L^2_\loc$.
\end{abstract}
\maketitle
\tableofcontents
%--------------------------------------------------------------------
\section{Introduction}
%--------------------------------------------------------------------
Recently, the second order regularity for parabolic $p$-Laplace type equations has been studied by H{\o}eg and Lindqvist \cite{hoegl20}, Dong, Peng, Zhang and Zhou \cite{dongpzz20}, and the authors \cite{fengps22}.
In this article, we consider a rather general class of parabolic equations
\begin{align}\label{geparabeq}
    u_t-|Du|^{\gamma}\big(\Delta u+(p-2)\Delta_\infty^Nu\big)=0
\end{align}
with $1<p<\infty$ and $-1<\gamma<\infty$, where 
\begin{align*}
    \ilN u:=\abs{Du}^{-2}\sum_{i,j=1}^nu_{x_i}u_{x_j}u_{x_ix_j}=\abs{Du}^{-2}\la Du,D^2uDu\ra =\abs{Du}^{-2}\il u
\end{align*}
denotes the normalized infinity Laplacian.
The equation contains the game theoretic or normalized $p$-parabolic equation and the divergence form standard $p$-parabolic equation as special cases. 
The equation is not uniformly parabolic or in divergence form except in special cases, and it can be highly degenerate or singular in the gradient variable. Regularity for such equations has been recently studied for example by Imbert, Jin and Silvestre as well as Parviainen and V\'azquez as discussed below.
The objective of this article is to develop a systematic approach to study the second order spatial regularity of viscosity solutions to \eqref{geparabeq}. In this approach proving the estimate reduces down to verifying that a certain coefficient matrix is positive definite.
For the further notation and the definition of viscosity solutions to \eqref{geparabeq}, we refer to Section \ref{sec:Main-results}. 

In \cite{fengps22}  we considered second order Sobolev regularity of the parabolic $p$-Laplace equation
\begin{equation} \label{eq:parabolic-p-Laplace}
    u_t-\Delta_p u=0
\end{equation}
where
$\Delta_pu:=\diverg(|Du|^{p-2}Du)$
is the $p$-Laplace operator. 
Notice that, in the special case $\gamma=p-2$, equation \eqref{geparabeq} can be formally, and also rigorously by \cite{juutinenlm01}, rewritten as \eqref{eq:parabolic-p-Laplace}.   
One of the key tools is the fundamental inequality (the name stems from Dong, Peng, Zhang and Zhou  \cite{dongpzz20} for a related inequality)
\begin{align}\label{eq:fundineq}
    |Du|^4|D^2u|^2
    \geq 
    2|Du|^2|D^2uDu|^2+\frac{(|Du|^2\Delta u-\il u)^2}{n-1}-(\il u)^2
\end{align}
which holds for any smooth function $u$ as shown by Sarsa in \cite{sarsa20}. Curiously, in \cite{fengps22} it was sufficient to use the above inequality in a simpler form just estimating $(|Du|^2\Delta u-\il u)^2\ge 0$  on the right hand side. With the general equation in this paper, we use the inequality in the full generality. 
A natural approach to obtain second order Sobolev estimates is to differentiate (\ref{geparabeq}), multiply the equation with suitable quantities containing gradients, and manipulate in a suitable way. Thus, among other terms, one can obtain terms in divergence form, which can be controlled. In the case of (\ref{eq:parabolic-p-Laplace}), one then uses (\ref{eq:fundineq}) in a simple form as explained above and thus gets an upper bound for a quantity containing second derivatives.  Part of the difficulty in dealing with the general equation instead of the $p$-parabolic equation stems from the fact that this approach gives rise to the mixed terms of the type
\begin{align*}
   |Du|^{-\gamma}u_t\ilN u
\end{align*}
which are difficult to handle. 

Another difficulty arises from the fact that of course $u$ is not known to be smooth  a priori when differentiating the equation, and negative powers of the gradient are problematic as the gradient might vanish. A natural approach to these problems is regularizing the equation by adding a small regularization parameter, which removes the singularity. Unfortunately, when differentiating the regularized equation, one gets another set of problematic terms that no longer match the terms in the fundamental inequality. Treating these terms is a subtle issue, and we need to guarantee that a sum of certain terms remain nonnegative by carefully analyzing explicit coefficients of the terms.

In order to analyze the nonnegativity of the problematic terms and their coefficients systematically, we develop several techniques. We interpret the terms and their coefficients, as a quadratic form and derive a range condition for the parameters from the positive definiteness condition of this quadratic form. In order to improve the range obtained in this way, we use a hidden divergence structure. Indeed, suitable mixed terms can actually be written in a divergence form, 
and thus by adding such terms, we can manipulate the coefficients at the cost of adding divergence form terms that can be estimated.  

Some steps, in particular checking that the quadratic form is positive definite,  of the above plan when written down explicitly are quite complicated, and thus for the convenience of the reader we first provide a formal calculation in Section \ref{sec:smcase}, where we assume that the solution is smooth and the gradient nonvanishing. In this case, the above plan gives an optimal (optimality is discussed in Example \ref{counterexample})  a priori estimate (Proposition \ref{prop:smooth}),
\begin{equation}
\label{eq:smooth-est}
       \int_{Q_r}\Big|D(|Du|^{\frac{p-2+s}{2}} Du)\Big|^2dxdt
    \leq
    \frac{C}{r^2}\Big(
    \int_{Q_{2r}}|Du|^{p+s}dxdt
    +
    \int_{Q_{2r}}|Du|^{p+s-\gamma}dxdt
    \Big)    
\end{equation}
in the  range
\begin{align*}
    1<p<\infty, \quad -1<\gamma<\infty\quad \text{and} \quad n\ge 2.
\end{align*}
with the range condition
\begin{align} 
\label{range:expected-2}
    s>\max\Big\{-1-\frac{p-1}{n-1},\gamma+1-p\Big\}.
\end{align}
The left hand side in the above estimate is of the same form as the estimate in \cite{fengps22}. In particular, we may set  $s=2-p$, $s=0$ and  $s=p-2$ giving 
\begin{align*}
    D^2u,\quad  D(\abs{Du}^{\frac{p-2}{2}}Du)\quad \text{ and }\quad D(\abs{Du}^{p-2}Du)
\end{align*} 
as special cases.

Perhaps surprisingly, removing the smoothness assumption and the assumption on the nonvanishing gradient by using the regularized equation turns out to be a problem. In particular, the additional terms resulting from the regularization add to the technical complication of showing that the quadratic form is positive definite. To reduce technical complication partly for expository reasons, we have decided to restrict ourselves to the case $n=2$ in the regularized case. In this context we obtain the following result.
\begin{theorem} \label{resultforwholep}
Let $n=2$. Let $u\colon\Omega_T\rightarrow \R$ be a viscosity solution to the general $p$-parabolic equation \eqref{geparabeq}. If $p$ and $\gamma$ satisfy one of the following conditions:
\begin{itemize}
    \item[(i)] $1<p\leq5$ and $-1<\gamma<1$; or
    \item[(ii)] $1<p<\infty$ and $-1<\gamma<\sqrt{2}-\frac12$,
\end{itemize}
then $D^2u$ exists and belongs to $L^2_{\rm loc}(\Omega_T)$. Moreover, we have the estimate
\begin{align} \label{eq:BoundforSecondDerivatives}
  \int_{Q_{ r}}|D^2u|^2dxdt
 \leq  \frac{C}{r^2}\Big( \int_{ Q_{2r}}  |Du|^2 dxdt + \int_{ Q_{2r}}|Du|^{ 2-\gamma } dxdt   \Big),
\end{align}
where $C=C(p,\gamma)>0$  and $Q_r\subset Q_{2r}\Subset\Om _T$ are concentric parabolic cylinders.
\end{theorem}

This also implies that time derivative exists as an $L^2$-function, which is not evident directly by the definition.
\begin{corollary}[Time derivative] \label{cor:TimeDerivative-w22}
Let $n=2$. Let $u\colon\Om _T\to\R$ be a  viscosity  solution to the general $p$-parabolic equation \eqref{geparabeq}. If $p$ and $\gamma$ satisfy one of the following conditions:
\begin{itemize}
    \item[(i)] $1<p\leq5$ and $0\leq\gamma<1$; or
    \item[(ii)] $1<p<\infty$ and $0\leq\gamma<\sqrt{2}-\frac12$,
\end{itemize}
then the time derivative $u_t$ exists as a function and
$u_t\in L^2_{\loc}(\Om _T)$.
\end{corollary}

At least to some extent the range condition in Theorem \ref{resultforwholep} is an artifact as we explain later. It would be interesting to know whether the theorem is valid in the whole range of parameters. 

Next we review the known regularity results of equation \eqref{geparabeq} and explain how our results fit into the existing literature.
If $\gamma=p-2$, then equation \eqref{geparabeq} is the parabolic $p$-Laplace equation \eqref{eq:parabolic-p-Laplace}. For the regularity theory of weak solutions to \eqref{eq:parabolic-p-Laplace} we refer to the monograph of DiBenedetto \cite{dibenedetto93}. In particular, if $u$ is a continuous weak solution to \eqref{eq:parabolic-p-Laplace}, then $u\in C^\alpha_\loc$ and $Du\in C^\beta_\loc$ for some $0<\alpha,\beta<1$. 

Moreover, Lindqvist  \cite{lindqvist08} showed in the degenerate case $2<p<\infty$ that 
$$
D(|Du|^{\frac{p-2}{2}}Du)\in L^2_\loc,
$$ and further that 
$$
D(|Du|^{p-2}Du)\in L^{\frac{p}{p-1}}_\loc.
$$ 
The singular case is treated in \cite{lindqvist17}.
The results then imply the existence of time derivative $u_t$ as a function in suitable spaces similar to Corollary~\ref{cor:TimeDerivative-w22}.  In the case of the obstacle problem the existence of the time derivative was established in \cite{lindqvist12}.
Dong, Peng, Zhang and Zhou \cite{dongpzz20} gave a proof that $D^2u\in L^2_\loc$ with a sharp range $1<p<3$. 
This range of $p$ can be recovered from assumption (i) of Theorem \ref{resultforwholep}. In the global case, estimates for $D(|Du|^{p-2}Du)$ have been derived by Cianchi and Maz'ya in \cite{cianchim19}. 

If $\gamma=0$, equation \eqref{geparabeq} is the normalized parabolic $p$-Laplace equation
\begin{equation} \label{eq:parabolic-p-Laplace-normalized}
    u_t-\plN u=0
\end{equation}
where $\plN u:=\Delta u+(p-2)\Delta_\infty^Nu$ is the normalized or game theoretic $p$-Laplace operator. This equation arises from a two-player stochastic game with a fixed running time, see Manfredi, Parviainen and Rossi \cite{manfredipr10}, or from image processing, see Does \cite{does11}. Banerjee and Garofalo \cite{banerjeeg13,banerjeeg15} studied the potential theoretic aspects and boundary regularity of the normalized p-Laplacian evolution. These papers also contain Lipschitz regularity results for solutions to the normalized $p$-parabolic equation. 
The regularity method in  \cite{manfredipr10} is global whereas in \cite{parviainenr16} a local game theoretic method is applied in this context.
Later Jin and Silveste \cite{jins17} established $C_{\loc}^{1,\alpha}$-regularity in space and  $C_{\loc}^{0,\frac{1+\alpha}{2}}$-regularity in time. In  \cite{hoegl20}, H{\o}eg and Lindqvist studied the second order Sobolev regularity for the normalized $p$-parabolic equation and showed that when $\frac65<p<\frac{14}5$, the second order spatial derivatives $D^2u$ and the time derivative $u_t$ belong to $L^2_\loc$. Moreover, they also proved that when $1<p<2$, $u_t$ also belongs to $L^2_\loc$. In \cite{attouchip18}, $C_{\loc}^{1,\alpha}$-regularity was established to the normalized $p$-parabolic equation with a source term.  The work of Dong, Peng, Zhang and Zhou \cite{dongpzz20} also applies to the normalized $p$-parabolic equation; in this case they obtained $D^2u\in L^{2+\delta}_\loc$ and $u_t\in L^{2+\delta}_\loc$ for some $\delta>0$ if $1<p<3+\frac{2}{n-2}$. The key result of  \cite{dongpzz20} with $\delta=0$ can be recovered from assumption (ii) of Theorem \ref{resultforwholep}. Recently Andrade and Santos \cite{andrades22} established improved Sobolev regularity estimates when $p$ is close to 2. 

As stated, \eqref{geparabeq} is in non-divergence form and can be highly degenerate or singular. Thus even defining viscosity solutions in such a way that existence and uniqueness can be obtained becomes a nontrivial issue. This was done by Ohnuma and Sato in \cite{ohnumas97}, see also Giga's monograph \cite{giga06}. For viscosity solutions to the general equation \eqref{geparabeq}, where $1<p<\infty$ and $-1<\gamma<\infty$ are allowed to be independent of each other, Imbert, Jin and Silvestre \cite{imbertjs19} proved in particular that $Du\in C^\alpha_\loc$ for suitable $0<\alpha<1$. In \cite{parviainenv20}, Parviainen and V\'azquez  established Harnack's inequality and asymptotic behaviour by using the fact that for radial solutions equation \eqref{geparabeq} is equivalent to a divergence form equation but in fictitious dimension. 
 Attouchi \cite{attouchi20} in the degenerate case and Attouchi-Ruosteenoja \cite{attouchir20} in the singular case established  spatial $C^{1,\alpha}_\loc$-regularity for an equation of type \eqref{geparabeq} but with a source term. The elliptic Harnack's inequality in the singular range was obtained in \cite{kurkinenps}. 

This article is organized as follows. 
In Section \ref{sec:Main-results} we provide the necessary preliminaries.
In Section \ref{sec:Plan} we explain 
the ideas of the proof of Theorem \ref{resultforwholep}.
In Section \ref{sec:Toolbox} we state several auxiliary lemmas needed in the proofs, including the fundamental inequality \eqref{eq:fundineq}.
Sections \ref{sec:smcase} and \ref{sec:full} are parallel to each other. In the former, we provide the formal calculation. In the latter, we provide a similar calculation in a regularized setting, which eventually yields Theorem \ref{resultforwholep}.
In Section \ref{sec:ProofofMain} we prove Theorem \ref{resultforwholep} and Corollary \ref{cor:TimeDerivative-w22}. Some of the proofs for the technical lemmas are postponed to the appendix.
%--------------------------------------------------------------------
\section{Preliminaries} \label{sec:Main-results}
%--------------------------------------------------------------------
We use the following notation. 
Let $\Om\subset \Rn$, $n\ge 2$, be a domain and define the cylinder 
$$ \Om_T:=\Om\times(0,T). $$ 
If $U$ is compactly contained in $\Om$, i.e.\ $U\subset \Om$ and the closure of $U$ is a compact subset of  $\Om$, we write  $U \Subset \Om$. For $0<t_1<t_2<\infty$, we set
$$ U_{t_1,t_2}:=U\times (t_1,t_2). $$
Moreover, we will use parabolic cylinders of the form 
$$ Q_r(x_0,t_0):=B_r(x_0)\times(t_0-r^2,t_0], $$
where $B_r(x_0)$ denotes the open ball with radius $r>0$ and center point $x_0\in\Om$.
When no confusion arises, we may drop the reference point $(x_0,t_0)$ and write $Q_{r }$. 

Given a function $u=u(x,t)$ of point $x\in\Rn$ and time $t>0$, the spatial gradient of $u$ is denoted by 
$Du=(u_{x_1},\ldots,u_{x_n})$, and the time derivative by $u_t$. The Hessian matrix of $u$ is denoted by $D^2u=(u_{x_ix_j})_{i,j=1}^n$. The Laplacian of $u$ is given by
\begin{align*}
\Delta u:=\sum_{i=1}^nu_{x_ix_i}    
\end{align*}
and the infinity Laplacian by
\begin{align*}
    \il u:=\sum_{i,j=1}^nu_{x_i}u_{x_j}u_{x_ix_j}=\la Du,D^2uDu\ra
\end{align*}
where $\la \cdot,\cdot\ra$ stands for the inner product in $\Rn$.
The normalized infinity Laplacian is denoted by 
\begin{align*}
\ilN u:=\frac{\il u}{|Du|^2}.
\end{align*}

We study viscosity solutions to the general $p$-parabolic equation
\begin{align}\label{eq:GeParabEq}
u_t-|Du|^{\gamma}\big(\Delta u+(p-2)\ilN u\big)=0\quad\text{in }\Omega_T,
\end{align}
where $1<p<\infty$ and $-1<\gamma<\infty$.
The definition of suitable viscosity solutions to \eqref{eq:GeParabEq}
requires some care because the operator may be singular. Nonetheless, a definition that fits our needs can be found in \cite{ohnumas97}. First set
\begin{equation}
\label{eq:F-def}
F(Du,D^2 u):=\abs{Du}^{\gamma}\big(\Delta u+(p-2)\Delta_{\infty}^{N} u\big)
\end{equation}
whenever $Du\neq 0$. We define $\mathcal F$ to be a set of functions $f\in C^{2}([0,\infty))$ such that
\[
\begin{split}
 f(0)=f'(0)=f''(0)=0,\ f''(r)>0 \text{ for all }r>0,
\end{split}
\]
and moreover we require for  $g(x):=f(\abs x)$ that
\[
\begin{split}
\lim_{x\to 0,x\neq 0}F(Dg(x),D^2g(x))=0.
\end{split}
\]
Further, let
\[
\begin{split}
\Sigma=\{ \sigma \in C^{1}(\R) \,:\, \sigma \text{ is even},\, \sigma(0)=\sigma'(0)=0,\text{ and }\sigma(r)>0 \text{ for all } r\neq 0 \}.
\end{split}
\]

\begin{definition} \label{def:admissible}
A function $\vp \in C^{2}(\Omega_T)$ is admissible if for any $(x_0,t_0)\in \Omega_T$ with $D\vp(x_0,t_0)=0$, there are $\delta>0$, $f\in \mathcal F$ and $\sigma \in \Sigma$ such that
\[
\abs{\vp(x,t)-\vp(x_0,t_0)-\vp_t(x_0,t_0)(t-t_0)}\le f(\abs{x-x_0})+\sigma(t-t_0)
\]
for all $(x,t)\in B_{\delta}(x_0)\times (t_0-\delta,t_0+\delta)$.
\end{definition}

If $D\vp\neq 0$, a $C^2$-function is automatically admissible.

\begin{definition} \label{eq:from-below}
We say that $\vp$ touches $u$ at $(x_0,t_0)\in \Omega_T$ (strictly) from below
if
\begin{enumerate}
\item[(1)] $u(x_0,t_0)=\vp(x_0,t_0)$, and
\item[(2)] $u(x,t)>\vp(x,t)$ for all $(x,t)\in \Omega_T$ such that $(x,t)\neq (x_0,t_0)$.
\end{enumerate}
\end{definition}

The definition for touching (strictly) from above is analogous.

\begin{definition} \label{def:visc}
A function $u:\Omega_T\to \R\cup \{\infty\}$ is a viscosity supersolution  to \eqref{eq:GeParabEq} if
\begin{enumerate}
\item[(i)] $u$ is lower semicontinuous,
\item [(ii)] $u$ is finite in a dense subset of $\Omega_T$,
\item [(iii)]for all admissible $\vp\in C^{2}(\Omega_T)$ touching $u$ at $(x_0,t_0)\in \Omega_T$ \ from below
\[
\begin{split}
\begin{cases}
\vp_t(x_0,t_0)-F(D\vp(x_0,t_0),D^2\vp(x_0,t_0))\ge 0 & \text{if }D\vp(x_0,t_0)\neq 0,\\
\vp_t(x_0,t_0)\ge 0 &  \text{if }D\vp(x_0,t_0)= 0.
\end{cases}
\end{split}
\]
\end{enumerate}
\end{definition}

The definition of a subsolution $u: \Omega_T\to \R\cup \{-\infty\}$ is analogous except that we require upper semicontinuity, touching from above, and we reverse the inequalities above: in other words if $-u$ is a viscosity supersolution.  If a continuous function is both a viscosity super- and subsolution, it is a {\sl viscosity solution.}

It is shown in \cite{juutinenlm01} that if $\gamma= p-2>-1$, then the above notion coincides with the notion of $p$-super/subparabolic functions, having a direct connection to the distributional weak super/subsolutions as well. Moreover, if $\gamma\ge 0$, then viscosity solutions can be defined in a standard way by using semicontinuous envelopes, see Proposition 2.2.8 in \cite{giga06}. 
%--------------------------------------------------------------------
\section{Plan of proof} \label{sec:Plan}
%--------------------------------------------------------------------
In this section we explain the idea of the proof of Theorem \ref{resultforwholep} and our plan of the proof.
%--------------------------------------------------------------------
\subsection{Derivation of a basic estimate}
%--------------------------------------------------------------------
In order to prove second order estimates, we first derive a key basic estimate \eqref{eq:RoadmapBasicIdentity} (or actually equality at this point). To this end, we regularize the original equation \eqref{geparabeq} and consider
\begin{equation} \label{eq:RoadmapRegularized}
    \ue_t-(\abs{D\ue}^2+\epsilon)^{\gamma/2}\Big(\Delta\ue+(p-2)\frac{\il\ue}{\abs{D\ue}^2+\epsilon}\Big)=0
\end{equation}
for small $\epsilon>0$. Solutions to this equation are smooth according to the standard theory.
We differentiate equation \eqref{eq:RoadmapRegularized} with respect to $x_k$, $k=1,\ldots,n$, and find that the spatial partial derivatives $\ue_{x_k}$, $k=1,\ldots,n$, solve the equation
\begin{equation} \label{eq:RoadmapRegularizedLinearization}
\begin{aligned}
    &(|D\ue|^2+\epsilon)^{\frac{p-2-\gamma}{2}}(\ue_{x_k})_t
    -\diverg\big((|D\ue|^2+\epsilon)^{\frac{p-2}{2}}AD\ue_{x_k}\big) \\
    &\quad
    +(p-2-\gamma)(|D\ue|^2+\epsilon)^{\frac{p-4-\gamma}{2}}\ue_t\la D\ue,D\ue_{x_k}\ra=0
\end{aligned}
\end{equation}
where
$$ A=I+(p-2)\frac{D\ue\otimes D\ue}{|D\ue|^2+\epsilon} $$
is a uniformly positive definite $n\times n$-matrix.
Here $I$ denotes the identity matrix.

We continue with the intention to study the derivatives of  $|Du|^{\frac{p-2+s}{2}}Du$; in particular the choice $s=2-p$ corresponds to $D^2u$. 
We multiply the differentiated equation 
\eqref{eq:RoadmapRegularizedLinearization} by $(|D\ue|^2+\epsilon)^{s/2}\ue_{x_k}$ and obtain
\begin{equation} \label{eq:RoadmapRegularizedLinearization-Mul}
\begin{aligned}
    &(|D\ue|^2+\epsilon)^{\frac{p-2-\gamma+s}{2}}\ue_{x_k}(\ue_{x_k})_t 
    -(|D\ue|^2+\epsilon)^{s/2}\ue_{x_k}\diverg\big((|D\ue|^2+\epsilon)^{\frac{p-2}{2}}AD\ue_{x_k}\big) \\
    &\quad
    +(p-2-\gamma)(|D\ue|^2+\epsilon)^{\frac{p-4-\gamma+s}{2}}\ue_t\la D\ue,D\ue_{x_k}\ra \ue_{x_k}=0.
\end{aligned}
\end{equation}
Using the chain rule
$$\ue_{x_k}(\ue_{x_k})_t =\frac{1}{2}\Big((\ue_{x_k})^2+\frac{\ez}{n}\Big)_t,$$
and summing \eqref{eq:RoadmapRegularizedLinearization-Mul} over $k=1,\ldots,n$ gives that 
\begin{equation} \label{eq:RoadmapRegularizedLinearization-Mul-Sum}
\begin{aligned}
    &\frac{\big((|D\ue |^2+\epsilon)^{\frac{p+s-\gamma}{2}}\big)_t}{p+s-\gamma} 
    -(|D\ue|^2+\epsilon)^{s/2}\sum_{k=1}^n\ue_{x_k}\diverg\big((|D\ue|^2+\epsilon)^{\frac{p-2}{2}}AD\ue_{x_k}\big) \\
    &\quad
    +(p-2-\gamma)(|D\ue|^2+\epsilon)^{\frac{p-2-\gamma+s}{2}}\ue_t\frac{\il \ue}{|D\ue|^2+\epsilon}=0.
\end{aligned}
\end{equation}
Observing that
\begin{align*}
    &\diverg\big((|D\ue|^2+\epsilon)^{\frac{p-2+s}{2}}AD^2\ue D\ue \big)\\
    =&\sum_{k=1}^n\diverg\Big(\big((|D\ue|^2+\epsilon)^{s/2} \ue_{x_k} \big)\big((|D\ue|^2+\epsilon)^{\frac{p-2}{2}}A D\ue_{x_k}\big)\Big)\\
    =&(|D\ue|^2+\epsilon)^{s/2}\sum_{k=1}^n\ue_{x_k}\diverg\big((|D\ue|^2+\epsilon)^{\frac{p-2}{2}}AD\ue_{x_k}\big)\\
    &+(|D\ue|^2+\epsilon)^{\frac{p-2+s}{2}} 
    \Big\{|D^2\ue|^2+(p-2+s)\frac{|D^2\ue D\ue|^2}{|D\ue|^2+\epsilon}
    +s(p-2)\frac{(\il \ue)^2}{(|D\ue|^2+\epsilon)^2}
    \Big\},
\end{align*}
we obtain the identity
\begin{equation} \label{eq:RoadmapBasicIdentity}
\begin{aligned}
    &(|D\ue|^2+\epsilon)^{\frac{p-2+s}{2}}
    \Big\{|D^2\ue|^2+(p-2+s)\frac{|D^2\ue D\ue|^2}{|D\ue|^2+\epsilon}
    +s(p-2)\frac{(\il \ue)^2}{(|D\ue|^2+\epsilon)^2} \\
    &\quad
    +(p-2-\gamma)(|D\ue|^2+\epsilon)^{-\gamma/2}\ue_t\frac{\il \ue}{|D\ue|^2+\epsilon} \Big\} \\
    &=
    \diverg\big((|D\ue|^2+\epsilon)^{\frac{p-2+s}{2}}AD^2\ue D\ue \big)
    -\frac{\big((|D\ue |^2+\epsilon)^{\frac{p+s-\gamma}{2}}\big)_t}{p+s-\gamma}.
\end{aligned}
\end{equation}
Here we assume that $s\neq\gamma-p$. This is not restrictive, because eventually such value of $s$ violates the resulting range condition  \eqref{range:expected-2} in any case. It is important that the terms on the right hand side are in divergence form and can thus be well estimated.  
An important step towards the desired result would be a pointwise inequality 
\begin{equation} \label{eq:Roadmap-Goal-Inequality}
    (|D\ue|^2+\epsilon)^{\frac{p-2+s}{2}}
    |D^2\ue|^2
    \lesssim
    \diverg\big((|D\ue|^2+\epsilon)^{\frac{p-2+s}{2}}AD^2\ue D\ue \big)
    -\frac{\big((|D\ue |^2+\epsilon)^{\frac{p+s-\gamma}{2}}\big)_t}{p+s-\gamma},
\end{equation} 
which then could be integrated to obtain the final result
and for this we need to estimate the excess terms on the left hand side of \eqref{eq:RoadmapBasicIdentity}. 
%-----------------------------------------------
\subsection{Formal calculation for smooth solutions with a nonvanishing gradient} \label{subs:Roadmap-smooth}
%--------------------------------------------------------------------
Compared to our earlier work \cite{fengps22} where we treated the case $\gamma=p-2$, we now have two extra difficulties for the general case $-1<\gamma<\infty$.
The first difficulty arises from the fourth term on the left hand side of \eqref{eq:RoadmapBasicIdentity}, that is,
\begin{equation} \label{eq:Mixed}
    (p-2-\gamma)(|D\ue|^2+\epsilon)^{-\gamma/2}\ue_t\frac{\il \ue}{|D\ue|^2+\epsilon}.
\end{equation}
Note that this mixed term vanishes if $\gamma=p-2$. In general we regard the term mixed in the sense that we cannot determine its sign by the sign of the coefficient $p-2-\gamma$.

We first discuss the difficulty of mixed terms in the formal case with $\eps=0$, and denote a solution by $u$. In this case, we assume in addition that $Du\neq0$.
As indicated above, we would like to estimate the excess term in \eqref{eq:RoadmapBasicIdentity} and obtain an estimate for $|Du|^{p-2+s}|D^2u|^2$ with the range \eqref{range:expected-2}. To this end, we write the fundamental inequality \eqref{eq:fundineq} in the form 
\begin{align*}
    2|D_T|Du||^2+\frac{(\Delta_T u)^2}{n-1}+(\ilN u)^2\le|D^2u|^2
\end{align*}
and employ it in identity \eqref{eq:RoadmapBasicIdentity} on the term $|D u|^{p-2+s}
    |D^2 u|^2$
to obtain that
\begin{equation} \label{eq:Roadmap-inequality-formal} 
\begin{aligned}
    &|Du|^{p-2+s}
    \Big\{\frac{1}{n-1}(\Delta_Tu)^2+(p+s)|D_T|Du||^2+(p-1)(s+1)(\ilN u)^2 \\
    &\quad
    +(p-2-\gamma)\abs{Du}^{-\gamma}u_t\ilN u\Big\} \\
    &\leq
    \diverg\big(|Du|^{p-2+s}AD^2uDu\big)
    -\frac{\big(\abs{Du}^{p+s-\gamma}\big)_t}{p+s-\gamma},
\end{aligned}
\end{equation}
where 
\begin{equation} \label{eq:Roadmap-decomp-1}
    \abs{D_T\abs{Du}}^2:=\frac{|D^2uDu|^2}{|Du|^2}-(\ilN u)^2
    \quad\text{and}\quad
    \Delta_T u:=\Delta u-\ilN u.
\end{equation}
Note that $|D_T|Du||^2\geq0$. Sometimes $\Delta_T u$ is called the normalized $1$-Laplacian for the obvious reason.

Except the mixed term that is the last term on the left hand side in \eqref{eq:Roadmap-inequality-formal}, the nonnegativity of other terms in the left hand side of \eqref{eq:Roadmap-inequality-formal} can be easily obtained by the restriction $s>-1$. 
In order to develop a systematic way of checking nonnegativity of the mixed term utilizing other terms, we use equation \eqref{geparabeq} to rewrite
$$\abs{Du}^{-\gamma}u_t\ilN u=\Delta_Tu\ilN u+(p-1)(\ilN u)^2, $$
and view the mixed term $\Delta_Tu\ilN u$ as a part of a quadratic form of $\Delta_T u$ and $\ilN u$.
That is, we consider
\begin{equation} \label{eq:Roadmap-naive-Q}
\begin{aligned}
    Q:
    &=\frac{1}{n-1}(\Delta_Tu)^2+(p-1)(p-1+s-\gamma)(\ilN u)^2+(p-2-\gamma)\Delta_Tu\ilN u \\
    &=:\la \bar{x},M\bar{x}\ra,
\end{aligned}
\end{equation}
where $\bar{x}:=(\Delta_Tu,\ilN u)^T\in\R^2$ and
$$ 
M=
\begin{bmatrix}
   \displaystyle{\frac{1}{n-1} }& \displaystyle{\frac{1}{2}}(p-2-\gamma)\\
   \displaystyle{\frac{1}{2}(p-2-\gamma)} & (p-1)(p-1+s-\gamma)
\end{bmatrix}
$$
is a symmetric $2\times 2$-matrix.

It turns out that in order to derive the desired estimate, it suffices to ensure along with few other conditions that the quadratic form $Q$ is strictly positive in $\R^2\backslash\{0\}$, that is, $M$ is positive definite. 
However, the range condition in \eqref{range:expected-2}
does not suffice to guarantee that the  positive definiteness of $Q$, hence we need to improve the estimate. 
We employ the following observation: If $q>1$, then 
\begin{equation}\label{eq:GoodStructureofqLaplacian}
   \begin{aligned}
    u_t |Du|^{q-2}\Delta_q^N u=u_t\diverg\big(|Du|^{q-2}Du\big)
    &=
    \diverg(u_t|Du|^{q-2}Du)-\frac{(|Du|^{q})_t}{q}
\end{aligned} 
\end{equation}
 holds for any smooth function $u$ with nonvanishing gradient.
In other words, the quantity on the left hand side is a `good term' with a hidden divergence structure.

It is easier to utilize this observation with inequality \eqref{eq:Roadmap-inequality-formal}, if we rewrite the right hand side of that inequality using equation \eqref{geparabeq}. To be more precise,
\begin{equation}\label{eq:Plan-good-structure}
\begin{aligned}
 \diverg&\big(|Du|^{p-2+s }AD^2uDu\big)
    -\frac{\big(\abs{Du}^{p +s-\gamma}\big)_t}{p +s-\gamma} \\
 &   =
\diverg\big(|Du|^{p-2+s }(D^2uDu-\Delta uDu)\big)
+u_t\diverg\big(|Du|^{p-2+s-\gamma}Du\big),
\end{aligned}    
\end{equation}
where the last term now matches with \eqref{eq:GoodStructureofqLaplacian} setting $q:=p+s-\gamma$. 
On the other hand, for a solution $u$, by equation \eqref{geparabeq}, and by the definition of normalized $q$-Laplacian $\Delta_q^Nu$, 
one has
\begin{align*}
u_t=|Du|^{\gamma}(\Delta_Tu+(p-1)\ilN u), \quad \text{and}\quad
\Delta_{p+s-\gamma}^Nu=\Delta_Tu+(p-1+s-\gamma)\ilN u 
\end{align*}
and thus
\begin{equation}\label{eq:to-quadratic-form}
\begin{aligned}
&u_t\diverg\big(|Du|^{p-2+s-\gamma}Du\big)\\
=&|Du|^{\gamma}\big(\Delta_Tu+(p-1)\ilN u\big)\cdot|Du|^{p-2+s-\gamma}\cdot\big(\Delta_Tu+(p-1+s-\gamma)\ilN u\big)  \\
 =&|Du|^{p-2+s }\Big\{(\Delta_Tu)^2+(2 p-2+s-\gamma) \Delta_Tu\ilN u+(p-1)(p-1+s-\gamma)(\ilN u)^2\Big\}. 
\end{aligned}
\end{equation}

The idea is to add $u_t\diverg\big(|Du|^{p-2+s-\gamma}Du\big)$ with a suitable weight on both sides of \eqref{eq:Roadmap-inequality-formal}: then by the above equation, it produces new coefficients on the left hand side that can be utilized later to get better range, and controllable terms on the right hand side by \eqref{eq:Plan-good-structure}. We also add another positive weight
by using 
\begin{equation}\label{eq:to-use-funineq}
\begin{aligned}
    |Du|^{p-2+s}&
    \Big\{|D^2 u|^2-(\Delta u)^2
    +(p-2+s)\frac{|D^2uDu|^2}{|Du|^2}-(p-2+s )\Delta u\ilN u\Big\} \\
    =&
    \diverg\big(|Du|^{p-2+s}(D^2uDu-\Delta uDu)\big)
\end{aligned}
\end{equation}
from Lemma \ref{lem:regularized-divergence-structure} below which holds for any smooth function with nonvanishing gradient.
This allows us to obtain simplified coefficients in intermediate steps. Thus we obtain 
\begin{equation} \label{eq:Roadmap-weighted-ineq-formal}
\begin{aligned}
    &|Du|^{p-2+s }
    \Big\{\Big(w_2- \frac{n-2}{n-1} w_1\Big)(\Delta_T u)^2 
    +w_1(p +s )\abs{D_T\abs{Du}}^2 \\
    &\quad
    +w_2(p-1)(p-1+s-\gamma)(\ilN u)^2+\big(w_2(2p-2+s-\gamma)-w_1(p +s )\big)\Delta_Tu\ilN u\Big\} \\
    &\leq
    w_1\diverg\big(|Du|^{p-2+s }(D^2uDu-\Delta uDu)\big)
    +w_2 u_t\diverg\big(|Du|^{p-2+s-\gamma}Du\big),
\end{aligned}
\end{equation}
which reduces to \eqref{eq:Roadmap-inequality-formal} if $w_1=1$ and $w_2=1$.
Calculations reveal that if the range condition \eqref{range:expected-2} holds, then the weights $w_1$ and $w_2$ can be adjusted so that the weighted quadratic form 
\begin{align*}
    &\Big(w_2-\frac{n-2}{n-1}w_1\Big)(\Delta_T u)^2 +w_2(p-1)(p-1+s-\gamma)(\ilN u)^2 \\
    &+\big(w_2(2p-2+s-\gamma)-w_1( p +s )\big)\Delta_Tu\ilN u
\end{align*}
is positive in $\R^2\backslash\{0\}$. This positivity in the formal case $\epsilon=0$ is shown in Lemma \ref{lem:smooth-key-lemma}. By Proposition \ref{prop:smooth}, this then implies the desired estimate
\begin{equation} \label{eq:roadmap-smooth-integral-estimate}
    \int_{Q_r}\abs{D(\abs{Du}^{\frac{p-2+s }{2}} Du)}^2dxdt
    \leq
    \frac{C}{r^2}\Big(
    \int_{Q_{2r}}|Du|^{p +s }dxdt
    +
    \int_{Q_{2r}}|Du|^{p+s-\gamma}dxdt
    \Big).
\end{equation}
\sloppy
Heuristically, in order to prove  the above estimate, and setting $s=2-p$ for simplicity, we could have left a small piece of $|D^2 u|^2$ when applying the fundamental inequality for \eqref{eq:Roadmap-inequality-formal}. Then the rest of the terms can be dropped by the above positivity result: in detail this is implemented in Lemma \ref{lem:ad-hoc} also for other values of $s$. The obtained pointwise estimate can then be integrated by parts along with a cutoff function to get Proposition \ref{prop:smooth}. 
%--------------------------------------------------------------------
\subsection{Solutions without smoothness assumptions and regularized equation} \label{sec:eps-plan}
%--------------------------------------------------------------------
The second difficulty, which is related to the regularization, is that the left hand side of \eqref{eq:RoadmapBasicIdentity} consists of regularized versions of second order derivative quantities,
$$ \frac{|D^2\ue D\ue|^2}{|D\ue|^2+\epsilon} \quad\text{and}\quad \frac{\il\ue}{|D\ue|^2+\epsilon}, $$
whereas employing the fundamental inequality \eqref{eq:fundineq} results in quantities like
$$ \frac{|D^2\ue D\ue|^2}{|D\ue|^2}\quad\text{and}\quad \ilN \ue. $$
This mismatch causes that some of the formal calculations do not work as such but have further complications: in particular positive definiteness of the quadratic form becomes an issue.

For a certain range of parameters, the main result
is obtained by a straightforward generalization of the formal calculation ($\epsilon=0$) in the previous section. 
However, in the process of extending the range, we consider
\begin{equation}\label{eq:Roadmap-S-with-errorterms-alpha-is-0}
\begin{aligned}
    S
    &:=
    w_1\diverg\big((|D\ue|^2+\epsilon)^{\frac{p-2+s}{2}}(D^2\ue D\ue-\Delta \ue D\ue)\big) \\
    &\quad
    +w_2\ue_t\diverg\big((|D\ue|^2+\epsilon)^{\frac{p-2+s-\gamma}{2}}D\ue\big) \\
    &\quad
    +w_3\epsilon\diverg\big((|D\ue|^2+\epsilon)^{\frac{p-2+s}{2}-1}(D^2\ue D\ue-\Delta \ue D\ue)\big) \\
    &\quad
    +w_4\epsilon\ue_t\diverg\big((|D\ue|^2+\epsilon)^{\frac{p-2+s-\gamma}{2}-1}D\ue\big),
\end{aligned}
\end{equation}
where $w_1,w_2,w_3,w_4\in\R$. 
Compared to the right hand side of \eqref{eq:Roadmap-weighted-ineq-formal}, or \eqref{eq:Roadmap-S-plane}, this sum has two additional terms with weights $w_3$ and $w_4$. The latter additional term has a hidden divergence structure, similarly to \eqref{eq:GoodStructureofqLaplacian}. These divergence structures can be used to adjust the coefficients on the left hand side of the estimate \eqref{eq:Roadmap-weighted-ineq-formal}, and thus to improve the range of parameters.
To be more precise, we denote
\begin{align}
\label{eq:theta-kappa}
    \theta:=\frac{|D\ue|^2}{|D\ue|^2+\epsilon}
\quad\text{and}\quad
\kappa:=1-\theta=\frac{\epsilon}{|D\ue|^2+\epsilon},
\end{align}
and obtain 
$$ \frac{|D^2\ue D\ue|^2}{|D\ue|^2+\epsilon}=\theta \abs{D\abs{D\ue}}^2 \quad\text{and}\quad \frac{\il\ue}{|D\ue|^2+\epsilon}=\theta\ilN \ue. $$
The second mixed term of \eqref{eq:Roadmap-S-with-errorterms-alpha-is-0}
can also be written as a part of the quadratic form as follows 
\begin{equation}\label{eq:to-quadratic-form-error}
\begin{aligned}
&\ez\ue_t\diverg\big((\abs{D\ue}^2+\ez)^{\frac{p-2+s-\gamma}{2}-1}D\ue\big)\\
 =&\theta(\abs{D\ue}^2+\ez)^{\frac{p-2+s }{2}}\Big((\Delta_T\ue)^2+\big((2p-6+s-\gamma)\theta+2 \big)\Delta_T\ue\ilN \ue\\
 &+\big((p-2)\theta+1\big)\big((p-4+s-\gamma)\theta+1 \big)(\ilN \ue)^2\Big)
\end{aligned}
\end{equation}
with weight $w_4$,
where 
\begin{align*}
\ue_t=(|D\ue|^2+\epsilon)^{\gamma/2}\Big(\Delta_T \ue+\big((p-2)\theta+1\big)\ilN\ue\Big)
\end{align*}
by using the regularized equation and recalling the shorthand notation
$\Delta_T \ue:=\Delta \ue-\ilN \ue.$
This will give rise to new coefficients and thus to a better range condition. 

In order to produce new coefficients on the left hand side of \eqref{eq:RoadmapBasicIdentity}, especially for the second order term $(\abs{D\ue}^2+\ez)^{\frac{p-2+s}{2}}\abs{D^2\ue}^2$, and also to improve the range of the parameters,
we add another divergence structure
\begin{equation}\label{eq:to-use-funineq-error}
\begin{aligned}
    &\ez\diverg\big((\abs{D\ue}^2+\ez)^{\frac{p-2+s}{2}-1}(D^2\ue D\ue-\Delta \ue D\ue)\big)\\
    =&\theta(\abs{D\ue}^2+\ez)^{\frac{p-2+s}{2}}
    \Big\{|D^2 \ue|^2-(\Delta \ue)^2
    +(p-4+s)\theta \abs{D\abs{D\ue}}^2\\
    &-(p-4+s)\theta \Delta \ue  \ilN\ue \Big\}.
\end{aligned}
\end{equation}
 Also observe that the above choice of the power $(p-2+s)/2-1$ will be useful in the proof of Lemma \ref{lem:lemma-for-final-estimate} when deriving an upper bound for the left hand side of the estimate, after integration by parts where we estimate $\eps/(|D\ue|^2+\epsilon)\le 1$ and thus the additional $-1$ in the power gets canceled out. Besides, the error terms obtained in Lemmas 4.4 and 4.7 in \cite{dongpzz20} can be seen as special cases of the error terms above.

Then combining \eqref{eq:to-quadratic-form}, \eqref{eq:to-use-funineq}, \eqref{eq:to-quadratic-form-error} and \eqref{eq:to-use-funineq-error} together with definition \eqref{eq:Roadmap-S-with-errorterms-alpha-is-0} of $S$, 
we get
\begin{align}
\label{eq:key-estimate-plan}
   &(|D\ue|^2+\epsilon)^{\frac{p-2+s}{2}}
    \Big\{
    c_1|D^2\ue|^2+c_2|D_T|D\ue||^2
    +(c_3-c_1)(\Delta_T\ue)^2 \\
    &\quad
    +\Big((c_3+c_4)\big((p-2)\theta+1\big)-c_1\Big)(\ilN \ue)^2 \\
    &\quad
    +\Big(c_3\big((p-2)\theta+1\big)+(c_3+c_4)-(2c_1+c_2)\Big) \Delta_T\ue\ilN \ue 
    \Big\}= S,
\end{align}
where $c_1,c_2,c_3$ and $c_4$ depend on $w_1,w_2,w_3, w_4$ and $\theta$ as computed in detail in Section \ref{sec:computations}. Then we again use the fundamental inequality on part of $c_1|D^2\ue|^2$ and find such weights $w_1,w_2,w_3$ and $w_4$ that 
the last three terms on the left hand side can be interpreted as a positive definite quadratic form and thus removed.  
Finally, $S$ on the right hand side can be multiplied by a cutoff function and integrated by parts to get the final estimate. However, the nonnegativity can only be checked in certain ranges, since it needs to hold uniformly for all $\theta\in [0,1)$.
%--------------------------------------------------------------------
\section{Hidden divergence structures, the key estimate and auxiliary lemmas} \label{sec:Toolbox}
%--------------------------------------------------------------------
In this section we prove several auxiliary tools. The lemmas in this section will be used to prove estimates for both $\ue$, that solves \eqref{eq:RoadmapRegularized} with $\epsilon>0$, and $u$, that solves \eqref{eq:RoadmapRegularized} with $\epsilon=0$ and $Du\neq 0$. Therefore we state the lemmas in such a generality that applies to both of these cases. 
%--------------------------------------------------------------------
\subsection{Hidden divergence structures}
%--------------------------------------------------------------------
In this subsection we gather some useful facts about generic smooth functions.
First, if $u\colon \Omega_T\to\R$, $\Omega_T\subset \R^{n+1}$ is a smooth function, then $|Du|$ is locally Lipschitz continuous and thus, by Rademacher's theorem, differentiable almost everywhere on each time slice. 
Here and in similar occurrences in what follows, we write that $D|Du|$ exists \emph{almost everywhere in space}.

Note that if $(x_0,t_0)\in \Omega_T$ is a space-time point where $|Du|$ is differentiable and $Du(x_0,t_0)=0$, then $D|Du|(x_0,t_0)=0$. Indeed, if we had $D|Du|(x_0,t_0)\neq 0$, then we could find a point $\xi\in \Omega\times\{t_0\}$ (close to $(x_0,t_0)$) such that $|Du|(\xi)<0$, which is obviously impossible.
On the other hand, if $Du(x_0,t_0)\neq 0$ for some $(x_0,t_0)\in \Omega_T$, then $|Du|$ is differentiable at $(x_0,t_0)$ and
$$ D|Du|(x_0,t_0)=\frac{D^2u(x_0,t_0)Du(x_0,t_0)}{|Du(x_0,t_0)|}. $$

For each point in $\Omega_T$ where $Du\neq0$, we fix an orthonormal basis of $\Rn$, $\{e_1,\ldots,e_n\}$,
such that $e_n=\frac{Du}{|Du|}$. Hence we have, for those points where $Du\neq0$,
$$ \frac{D^2uDu}{|Du|}
=\la e_1,D|Du|\ra e_1
+\ldots
+\la e_{n-1},D|Du|\ra e_{n-1}
+\left<\frac{Du}{|Du|},D|Du|\right>\frac{Du}{|Du|}. $$
For those points where $|Du|$ is differentiable, let us define the part of $D|Du|$ which is tangential to the spatial level sets of $u$ as
\begin{align*}
    D_T|Du|:=
    \begin{cases}
    \la e_1,D|Du|\ra e_1
    +\ldots
    +\la e_{n-1},D|Du|\ra e_{n-1} &\quad\text{if }Du\neq 0, \\
    0 &\quad\text{if }Du=0,
    \end{cases}
\end{align*}
and its orthogonal counterpart, the normalized infinity Laplacian, as
\begin{align*}
    \ilN u:=
    \begin{cases}
    \left< \frac{Du}{|Du|},D|Du|\right>=\frac{\il u}{|Du|^2} &\quad\text{if }Du\neq 0, \\
    0 &\quad\text{if }Du=0.
    \end{cases}
\end{align*}
We employ these notation to write
\begin{equation} \label{eq:OrthogonalRepresentation}
    |D|Du||^2=\abs{D_T|Du|}^2+(\ilN u)^2 \quad\text{a.e.\ in space in } \Omega_T,
\end{equation}
and
\begin{equation} \label{eq:OrthogonalLaplacian}
    \Delta_T u=\Delta u-\ilN u \quad\text{a.e.\ in } \Omega_T.
\end{equation}

\begin{lemma}[Fundamental inequality] \label{lem:FundamentalInequality}
Let $u\colon \Omega_T\to\R$ be a smooth function. Then 
\begin{align}\label{eq:tri-ineq}
    |D^2u|^2\geq 2|D_T|Du| |^2
    +\frac{(\Delta_T u)^2}{n-1}
    +(\ilN u )^2\quad\text{a.e.\ in space in } \Omega_T.
\end{align}
If $n=2$, we have equality in the place of inequality.
\end{lemma}

For the proof of Lemma \ref{lem:FundamentalInequality}, we refer to \cite{sarsa20,fengps22}.

The following lemmas show that certain terms that first appear to be in non-divergence form, can actually  be expressed in a divergence form. On the other hand, these structures can be utilized in tuning the coefficients in the quadratic form as explained in Section \ref{sec:eps-plan}, and thus they improve the range we obtain. The first Lemma \ref{lem:regularized-divergence-structure} will mainly adjust the coefficient of the term $(\abs{D\ue}^2+\ez)^{\frac{p-2+s}{2}}\abs{D^2\ue}^2.$
The second divergence structure, Lemma \ref{lem:regularized-time-derivative-structure}, will produce certain new coefficients on the quadratic form as $Q$. 
The proofs of both of these lemmas are direct calculations.

\begin{lemma}[Hidden divergence structure 1] \label{lem:regularized-divergence-structure}
Let $u\colon \Omega_T\to\R$ be a smooth function. Then for any $\alpha\in\R$ and $\epsilon>0$,
\begin{equation} \label{eq:regularized-divergenge-structure}
\begin{aligned}
    &(|Du|^2+\epsilon)^{\alpha/2}
\Big\{|D^2u|^2-(\Delta u)^2
    +\alpha\frac{|D^2uDu|^2}{|Du|^2+\epsilon}-\alpha\Delta u\frac{\il u}{|Du|^2+\epsilon}\Big\}\\
    =&
    \diverg\big((|Du|^2+\epsilon)^{\alpha/2}(D^2uDu-\Delta uDu)\big).
\end{aligned}
\end{equation}
Furthermore, if $Du\neq 0$, then the above equality holds also for $\epsilon=0$.
\end{lemma}

\begin{proof}By the derivative rule of composite function, the right hand side
\begin{align*}
&\diverg\big((|Du|^2+\ez)^{\alpha/2}(D^2uDu-\Delta uDu )\big)\\
=& \big< D^2uDu-\Delta uDu , D\big((|Du|^2+\ez)^{\alpha/2} \big)\big>+ (|Du|^2+\ez)^{\alpha/2}\diverg ( D^2uDu-\Delta uDu  )\\
=&\big< D^2uDu-\Delta uDu , D\big((|Du|^2+\ez)^{\alpha/2} \big)\big>+ (|Du|^2+\ez)^{\alpha/2}\big(|D^2u|^2 -(\Delta u)^2 \big)\\
=&(|Du|^2+\ez)^{\alpha/2}\Big\{|D^2u|^2-(\Delta u)^2+\alpha\frac{|D^2uDu|^2}{|Du|^2+\ez}-\alpha\Delta u\frac{\il  u}{|Du|^2+\ez}\Big\},
\end{align*}
where 
\begin{align*}
\hspace{9 em}    D\big((|Du|^2+\ez)^{\alpha/ 2} \big)=\alpha(|Du|^2+\ez)^{\frac{\alpha-2} 2}D^2uDu.\hspace{10 em} \qedhere
\end{align*}
\end{proof}

The next lemma demonstrates that a mixed term can be written in a  divergence form. On the other hand by using equation \eqref{eq:RoadmapRegularized}, as explained in \eqref{eq:to-quadratic-form-error}, the mixed term adds up in the quadratic form, and thus adding such mixed terms can be used to improve the range.

\begin{lemma}[Hidden divergence structure 2] \label{lem:regularized-time-derivative-structure}
Let $u\colon \Omega_T\to\R$ be a smooth function. Then for any $\beta\in\R$ and $\epsilon>0$,
\begin{equation} \label{eq:regularized-time-derivative-structure}
\begin{aligned}
    &u_t(|Du|^2+\epsilon)^{\beta/2}
    \Big(\Delta u +\beta\frac{\il u}{|Du|^2+\epsilon}\Big) \\
    =&u_t\diverg\big((|Du|^2+\ez)^{\beta/2}Du\big) \\
    =&
    \begin{cases}
    \displaystyle\diverg\big(u_t(|Du|^2+\epsilon)^{\beta/2} Du\big)-\Big(\frac{(|Du|^2+\epsilon)^{\frac{\beta+2}{2}}}{\beta+2}\Big)_t
    &\quad\text{if }\beta\neq -2, \\
\displaystyle\diverg\big(u_t(|Du|^2+\epsilon)^{-1}Du\big)-\Big(\frac{\ln (|Du|^2+\epsilon)}{2}\Big)_t
    &\quad\text{if }\beta=-2.
    \end{cases}
\end{aligned}
\end{equation}
Furthermore, if $Du\neq 0$, then the above equality holds also for $\epsilon=0$.
\end{lemma}

\begin{proof}
We give the proof when $\beta\neq -2$, the second case is similar. By the derivative rule of composite function again, one has
\begin{align*}
&\diverg\big(u_t(|Du|^2+\epsilon)^{\beta/2} Du\big)-\Big(\frac{(|Du|^2+\epsilon)^{\frac{\beta+2}{2}}}{\beta+2}\Big)_t\\
=&u_t\diverg\big((|Du|^2+\epsilon)^{\beta/2} Du\big)+(|Du|^2+\epsilon)^{\beta/2} Du Du_t
-\Big(\frac{(|Du|^2+\epsilon)^{\frac{\beta+2}{2}}}{\beta+2}\Big)_t\\
=&u_t\diverg\big((|Du|^2+\epsilon)^{\beta/2} Du\big)\\
=&u_t\big<D\big((|Du|^2+\epsilon)^{\beta/2}\big), Du\big>+u_t (|Du|^2+\epsilon)^{\beta/2} \diverg(Du) \\
=&u_t(|Du|^2+\ez)^{\beta/2}\Big(\Delta u+\beta\frac{\il u}{|Du|^2+\ez}\Big).
\end{align*}
\end{proof}
%--------------------------------------------------------------------
For $\alpha\in\R$, we
denote the `first good divergence structure' as
\begin{equation} \label{eq:good-elliptic}
GD_1(\alpha):=\diverg\big((|Du|^2+\epsilon)^{\alpha/2}(D^2u Du-\Delta uDu)\big)
\end{equation}
and the `second good divergence structure' 
\begin{equation} \label{eq:good-parabolic}
\begin{aligned}
   GD_2(\alpha):=
    \begin{cases}
    \displaystyle\diverg\big(u_t(|Du|^2+\epsilon)^{\frac{\alpha-\gamma}{2}} Du\big)-\Big(\frac{(|Du|^2+\epsilon)^{\frac{\alpha-\gamma+2}{2}}}{\alpha-\gamma+2}\Big)_t
    &\quad\text{if }\alpha\neq \gamma-2, \\
    \displaystyle\diverg\big(u_t(|Du|^2+\epsilon)^{-1} Du\big)-\Big(\frac{\ln (|Du|^2+\epsilon)}{2}\Big)_t
    &\quad\text{if }\alpha=\gamma-2.
\end{cases}
\end{aligned}
\end{equation}

Then as explained in (\ref{eq:Roadmap-S-with-errorterms-alpha-is-0}), we consider the following weighted sum of these `good structures',
\begin{equation}\label{eq:S}
 \begin{aligned}
    S:=&w_1GD_1(p-2+s)+w_2GD_2(p-2+s)\\
&+\epsilon w_3GD_1(p-4+s)+\epsilon w_4GD_2(p-4+s) 
\end{aligned}   
\end{equation}
for some parameter $s\in\R$ and some weights $w_1,w_2,w_3,w_4\in\R$.  
Observe that taking into account Lemmas \ref{lem:regularized-divergence-structure} and \ref{lem:regularized-time-derivative-structure}, then $S$ introduced above coincides with $S$ in \eqref{eq:Roadmap-S-with-errorterms-alpha-is-0}, i.e.\ the notation is consistent. The reason for using the mixed term form in $S$ there was to emphasize the idea that we can improve the range by adding the mixed terms. To derive the final estimate, we need terms in the divergence form, and therefore this form was used in the above definition of $S$, but as stated they are equivalent.
%--------------------------------------------------------------------
\subsection{The key estimate}
\label{sec:computations}
%--------------------------------------------------------------------
As explained in \eqref{eq:key-estimate-plan}, $S$ represents the right hand side in our key estimate, and on the left we should have the second derivatives and a positive definite quadratic form. In this section, we derive the key estimate corresponding to \eqref{eq:key-estimate-plan} in detail.

We use Lemmas \ref{lem:regularized-divergence-structure} and \ref{lem:regularized-time-derivative-structure} to rewrite $S$ as a linear combination of time derivatives and second order spatial derivative quantities, similarly to the left hand side of \eqref{eq:RoadmapBasicIdentity}.
First recall shorthand notation $\theta$ and $\kappa$ from \eqref{eq:theta-kappa}
\begin{align*}
  \theta=\frac{|D\ue|^2}{|D\ue|^2+\epsilon}
\quad\text{and}\quad
\kappa=\frac{\epsilon}{|D\ue|^2+\epsilon},
\end{align*}
thus $0\leq \theta,\kappa\leq1$, $\theta+\kappa=1$ and 
\begin{align*}
        \frac{|D^2\ue D\ue|^2}{|D\ue|^2+\epsilon}=\theta \abs{D\abs{D\ue}}^2 \quad\text{and}\quad \frac{\il \ue}{|D\ue|^2+\epsilon}=\theta\ilN \ue.
\end{align*}
In particular, if $\epsilon=0$ and the gradient does not vanish, then $\theta\equiv1$ and $\kappa\equiv0$.
Next we recall the definition of $S$ from the above, and use the good divergence structures i.e.\ Lemma \ref{lem:regularized-divergence-structure} 
(with $\alpha=p-2+s \; \text{and}\; p-4+s$   ) and Lemma \ref{lem:regularized-time-derivative-structure} (with $\beta=p-2+s-\gamma\;\text{and}\; p-4+s-\gamma$). For a smooth solution and indeed for any smooth function, we have
\begin{align*}
   S
    =&
    w_1 (|D\ue|^2+\epsilon)^{\frac{p-2+s}{2}}
    \Big\{|D^2\ue|^2-(\Delta \ue)^2
    +(p-2+s)\Big(\frac{|D^2\ue D\ue|^2}{|D\ue|^2+\epsilon}- \Delta \ue\frac{\il \ue}{|D\ue|^2+\epsilon}\Big)\Big\}\\
    &+w_2\ue_t(|D\ue|^2+\epsilon)^{\frac{p-2+s-\gamma}{2}}
    \Big\{\Delta \ue +(p-2+s-\gamma)\frac{\il \ue}{|D\ue|^2+\epsilon}\Big\}\\
    &+\ez w_3(|D\ue|^2+\epsilon)^{\frac{p-4+s}{2}}
    \Big\{|D^2\ue|^2-(\Delta \ue)^2
    +(p-4+s)\Big(\frac{|D^2\ue D\ue|^2}{|D\ue|^2+\epsilon}- \Delta \ue\frac{\il \ue}{|D\ue|^2+\epsilon}\Big)\Big\}\\
    &+\ez w_4\ue_t(|D\ue|^2+\epsilon)^{\frac{p-4+s-\gamma}{2}}
    \Big\{\Delta \ue +(p-4+s-\gamma)\frac{\il \ue}{|D\ue|^2+\epsilon}\Big\}.
\end{align*}

Then by simplifying, we get
    \begin{align*}
        S=&(|D\ue|^2+\epsilon)^{\frac{p-2+s}{2}}
    \Big\{
    (w_1+w_3\kappa)\big(|D^2\ue|^2-(\Delta \ue)^2\big)\\
    &
    +\big(w_1(p-2+s)+w_3(p-4+s)\kappa\big)\theta(|D|D\ue||^2-\Delta \ue\ilN \ue) \\
    &
    +(w_2+w_4\kappa)(|D\ue|^2+\epsilon)^{-\gamma/2}\ue_t\Delta \ue\\
    &
    +\big(w_2(p-2+s-\gamma)+w_4(p-4+s-\gamma)\kappa\big)\theta(|D\ue|^2+\epsilon)^{-\gamma/2}\ue_t\ilN \ue
    \Big\}\\
   =&
   (|D\ue|^2+\epsilon)^{\frac{p-2+s}{2}}
    \Big\{
    c_1\big(|D^2\ue|^2-(\Delta \ue)^2\big)
    +c_2(|D|D\ue||^2-\Delta \ue\ilN \ue) \\
    &\quad
    +c_3(|D\ue|^2+\epsilon)^{-\gamma/2}\ue_t\Delta \ue
    +c_4(|D\ue|^2+\epsilon)^{-\gamma/2}\ue_t\ilN \ue
    \Big\}
\end{align*}
almost everywhere in $\Omega_T$, where 
\begin{align}\label{eq:c1}
\begin{cases}
    c_1=w_1+w_3\kappa,\qquad
    &c_2=\big(w_1(p-2+s)+w_3(p-4+s)\kappa\big)\theta,\\
    c_3=w_2+w_4\kappa,  \qquad     &c_4=\big(w_2(p-2+s-\gamma)+w_4(p-4+s-\gamma)\kappa\big)\theta.
\end{cases}
\end{align}
Observe that given $p$, $\gamma$ and $s$, if $\epsilon=0$, then $c_1,\ldots,c_4$ reduce to constants that only depend on $w_1$ and $w_2$, which shows that in smooth case by adjusting $w_1$ and $w_2$, we can get the desired estimate as explained in  \eqref{eq:Roadmap-weighted-ineq-formal}. 

By employing expressions \eqref{eq:OrthogonalRepresentation} and \eqref{eq:OrthogonalLaplacian}, we can write
\begin{equation} \label{eq:S-with-lemmas-and-decomp}
\begin{aligned}
    S
    &=
    (|D\ue|^2+\epsilon)^{\frac{p-2+s}{2}}
    \Big\{
    c_1|D^2\ue|^2+c_2|D_T|D\ue||^2 
    -c_1(\Delta_T\ue)^2
    -c_1(\ilN \ue)^2 \\
    &\quad
    -(2c_1+c_2)\Delta_T\ue\ilN \ue
    +c_3(|D\ue|^2+\epsilon)^{-\gamma/2}\ue_t\Delta_T\ue \\
    &\quad
    +(c_3+c_4)(|D\ue|^2+\epsilon)^{-\gamma/2}\ue_t\ilN \ue
    \Big\}
\end{aligned}
\end{equation}
almost everywhere in $\Omega_T$.
Next we use regularized equation \eqref{eq:RoadmapRegularized} to replace time derivatives $u_t$ in \eqref{eq:S-with-lemmas-and-decomp} with spatial derivatives. Thus we arrive to the key estimate 
for a smooth solution to the regularized equation (which is actually equality at this point) 
\begin{equation} 
\label{eq:replace-time-derivatives}
\begin{aligned}
    &(|D\ue|^2+\epsilon)^{\frac{p-2+s}{2}}
    \Big\{
    c_1|D^2\ue|^2+c_2|D_T|D\ue||^2
    +(c_3-c_1)(\Delta_T\ue)^2 \\
    &\quad
    +\big((c_3+c_4)P_\theta -c_1\big)(\ilN \ue)^2 \\
    &\quad
    +\big(c_3P_\theta+(c_3+c_4)-(2c_1+c_2)\big)\Delta_T\ue\ilN \ue
    \Big\}=S,
\end{aligned}
\end{equation}
where $$P_\theta :=(p-2)\theta+1 \in(0,\infty) $$
for the sake of brevity.
We rewrite this as 
\begin{equation} \label{eq:S-shortly-written}
\begin{aligned}
    (|D\ue|^2+\epsilon)^{\frac{p-2+s}{2}}
    \Big\{
    c_1|D^2\ue|^2+c_2|D_T|D\ue||^2 + R\Big\}=S
\end{aligned}
\end{equation}
where
\begin{align*}
    R
    &:=
    (c_3-c_1)(\Delta_T\ue)^2 
    +\big((c_3+c_4)P_\theta-c_1\big)(\ilN \ue)^2 \\
    &\quad
    +\big(c_3P_\theta+(c_3+c_4)-(2c_1+c_2)\big)\Delta_T\ue\ilN \ue
\end{align*}
is a quadratic form in variables $\Delta_T u$ and $\ilN u$. We rewrite $R$ as
$$ R=\la \bar{x},N\bar{x}\ra, $$
where $\bar{x}=(\Delta_T \ue,\ilN \ue)^T\in\R^2$ and $N\in\R^{2\times 2}$ is a symmetric matrix whose entries $N_{ij}$, $i,j=1,2$, are given by 
\begin{align*} 
\begin{cases}
    N_{11}&=c_2-c_1\\
    N_{12}=N_{21}&=
    \tfrac{1}{2}\big(c_3P_\theta+(c_3+c_4)
    -(2c_1+c_2)\big)\\
    N_{22}&=(c_3+c_4)P_\theta-c_1.
\end{cases}
\end{align*}
Note that
\begin{align*}
    \|N\|_{L^\infty(\Omega_T)}
    :=\sup\{|N(x,t)|:(x,t)\in \Omega_T\}
\end{align*}
where 
$$|N(x,t)|=\sqrt{\big(N_{11}(x,t)\big)^2+\big(N_{12}(x,t)\big)^2+\big(N_{21}(x,t)\big)^2+\big(N_{22}(x,t)\big)^2},$$
has an upper bound that only depends on $p$, $\gamma$ and $s$  by fixing $w_1,w_2,w_3$ and $w_4$.  
%--------------------------------------------------------------------
\subsection{Auxiliary lemmas}\label{subsec:Sum S}
%--------------------------------------------------------------------
In this subsection we state two technical lemmas that can be used to conclude our main integral estimate.

We want to apply the fundamental inequality, Lemma \ref{lem:FundamentalInequality}, to estimate $|D^2\ue|^2$ in \eqref{eq:S-shortly-written} from below to improve the range condition by using terms we obtain in this application.
However, the direct application will eliminate the full Hessian $|D^2\ue|^2$ that we want to estimate. We could leave a small fraction of $|D^2\ue|^2$ (like the method was first described at the end of Section \ref{sec:eps-plan} for simplicity) and apply the fundamental inequality only to a remaining part, but actually this will not be necessary: The next lemma  shows that already a seemingly weaker lower bound is sufficient. This will simplify the exposition. 

\begin{lemma} \label{lem:ad-hoc}
Let $\ue\colon\Omega_T\to\R$ be a smooth solution to \eqref{eq:RoadmapRegularized}, $S$ as in \eqref{eq:S}, $c_1$ as in \eqref{eq:c1}, and $\eps\ge 0$.
If $\epsilon=0$, we assume in addition that $D\ue\neq0$.
 Suppose that we can select  $w_1,w_2,w_3,w_4\in\R$ such that 
$c_1=c_1(n,p,\gamma,s,w_1,w_2,w_3,w_4)>0$, $c=c(n,p,\gamma,s,w_1,w_2,w_3,w_4)>0$ and
\begin{align}\label{eq:lower-bound-S}
(|D\ue|^2+\epsilon)^{\frac{p-2+s}{2}}\Big\{c|D_T|D\ue||^2+Q\Big\}\le S \quad\text{a.e.\ in space in }\Omega_T,
\end{align}
where 
$$ Q=\la \bar{x},M\bar{x}\ra $$
with $\bar{x}=(\Delta_T\ue,\ilN \ue)^T\in\R^2$ and a  uniformly bounded positive definite (with a uniform constant) symmetric matrix $M=M(n,p,\gamma,s,w_1,w_2,w_3,w_4)\in \R^{2\times 2}$. Then there is
$\lambda=\lambda(n,p,\gamma,s,w_1,w_2,w_3,w_4)>0$ 
such that
\begin{align*}
    \lambda (|D\ue|^2+\epsilon)^{\frac{p-2+s}{2}}|D^2\ue|^2\le S \quad\text{a.e.\ in space in }\Omega_T. 
\end{align*}
\end{lemma}

\begin{proof}
Recall that
\begin{align*}
    S=w_1GD_1(p-2+s)+w_2GD_2(p-2+s)+\epsilon w_3GD_1(p-4+s)+\epsilon w_4GD_2(p-4+s) 
\end{align*}
which, as pointed out in \eqref{eq:S-shortly-written}, can be written as
\begin{equation}
\begin{aligned}
    S
    &=
    (|D\ue|^2+\epsilon)^{\frac{p-2+s}{2}}
    \Big\{
    c_1|D^2\ue|^2+c_2|D_T|D\ue||^2+\la \bar{x},N\bar{x}\ra
    \Big\}
\end{aligned}    
\end{equation}
almost everywhere in space $\Omega_T$, where
$\bar{x}$ and $N$ are as in \eqref{eq:S-shortly-written}. Observe that we utilized equation \eqref{eq:RoadmapRegularized} at this step to get rid of the time derivatives. 

For any $\lambda\in(0,1)$, we write
$$ S=\lambda S+(1-\lambda)S $$
and use the assumption \eqref{eq:lower-bound-S} to estimate $(1-\lambda)S$ from below. We end up with 
\begin{align*}
    S
    &\geq
    (|D\ue|^2+\epsilon)^{\frac{p-2+s}{2}}
    \Big\{
    \lambda c_1|D^2\ue|^2
    +\big(c+\lambda(c_2-c)\big)|D_T|D\ue||^2 \\
    &\quad
    +\left< \bar{x},\big(M+\lambda(N-M)\big)\bar{x}\right> 
    \Big\}.
\end{align*}
We claim that we can select $\lambda>0$ such that $c+\lambda(c_2-c)\geq 0$ and $M+\lambda(N-M)$
is a positive definite matrix. 
Indeed, since $c>0$, then
\begin{align*}
    c+\lambda(c_2-c)
    \geq 
    c-\lambda\|c_2-c\|_{L^\infty(\Omega_T)}>0,
\end{align*}
uniformly if 
$\lambda=\lambda(n,p,\gamma,s,w_1,w_2,w_3,w_4)>0$
is small enough.
Next we recall that the boundedness and positive definiteness of $M$ implies 
\begin{align} \label{eq:conditions-for-M}
    \|M\|_{L^\infty(\Omega_T)}\leq C,
    \quad\text{and}\quad
    M_{11}\geq c
    \quad\text{and}\quad
    \det(M)\geq c \quad\text{in }\Omega_T
\end{align}
by Sylvester's criterion and choosing small enough $c>0$.
For the positive definiteness of the matrix $M+\lambda(N-M)$ we can use Sylvester's criterion again and check that the leading principal minors are positive if $\lambda>0$ is small enough.
The first principal minor is the upper-left corner entry, i.e.
\begin{align*}
    \big(M+\lambda(N-M)\big)_{11}
    =M_{11}+\lambda(N_{11}-M_{11})
    \geq c-\lambda(\|N\|_{L^{\infty}(\Omega_T)}+\|M\|_{L^\infty(\Omega_T)}),
\end{align*}
and the second principal minor is the determinant, i.e.
\begin{align*}
    \operatorname{det}\big(M+\lambda(N-M)\big)
    &=
    \operatorname{det}(M)
    +\lambda\big(M_{11}N_{22}+M_{22}N_{11}-2M_{12}N_{12}-2\det(M)\big) \\
    &\quad
    +\lambda^2\operatorname{det}(N-M) \\
    &\geq
    c-2\lambda\left(2\|M\|_{L^\infty(\Omega_T)}\|N\|_{L^\infty(\Omega_T)}+\|M\|_{L^\infty(\Omega_T)}^2\right)-\lambda^2\|N-M\|_{L^\infty(\Omega_T)}^2 \\
    &\geq
    c-4\lambda\left(\|M\|_{L^\infty(\Omega_T)}+\|N\|_{L^\infty(\Omega_T)}\right)^2.
\end{align*}
Hence we choose $\lambda$ such that
\begin{align*}
     0<\lambda <
    \min\Big\{1,  
    \frac{c}{\|c_2-c\|_{L^\infty(\Omega_T)}},&
    \frac{c}{\|N\|_{L^{\infty}(\Omega_T)}+\|M\|_{L^\infty(\Omega_T)}},\\
    &
    \frac{c}{4\left(\|M\|_{L^\infty(\Omega_T)}+\|N\|_{L^\infty(\Omega_T)}\right)^2}
    \Big\}.
\end{align*}
Since we have now proven the nonnegativity of the excess terms, the result follows.
\end{proof}

The following lemma shows that we can derive the desired integral estimate from the pointwise lower bound. 
The proof uses rather standard techniques and is based on localization with a suitable cutoff function and then integration by parts. For the convenience of the reader, we give the details in the appendix. 

\begin{lemma} \label{lem:lemma-for-final-estimate}
Let $\ue\colon\Omega_T\to\R$ be a smooth solution to \eqref{eq:RoadmapRegularized}, and $S$ as in \eqref{eq:S}.
If $\epsilon=0$, we assume in addition that $D\ue\neq0$.
Suppose that we can find weights $w_1,w_2,w_3,w_4\in\R$ such that
\begin{equation} \label{eq:lower-bound-for-S}
    \lambda(|D\ue|^2+\epsilon)^{\frac{p-2+s}{2}}|D^2\ue|^2\le S \quad\text{a.e.\ in space in }\Omega_T,
\end{equation}
for some constant $\lambda=\lambda(n,p,\gamma,s,w_1,w_2,w_3,w_4)>0$.
If $s\neq \gamma-p$, then for any concentric parabolic
cylinders $Q_r\subset Q_{2r}\Subset\Omega_T$ with center point $(x_0,t_0)\in\Om_T$, we have the estimate
\begin{equation} \label{eq:lemma-for-final-estimate}
    \begin{aligned}
    \int_{Q_r}&\abs{D\big((|D\ue|^2+\epsilon)^{\frac{p-2+s}{4}}D\ue\big)}^2dxdt\\
    \leq&
    \frac{C}{r^2}\Big(
    \int_{Q_{2r}}(|D\ue|^2+\epsilon)^{\frac{p-2+s}{2}}|D\ue|^2dxdt 
    +\int_{Q_{2r}}(|D\ue|^2+\epsilon)^{\frac{p+s-\gamma}{2}}dxdt\Big) \\
    &+\epsilon\Big(\frac{C}{r^2}\int_{Q_{2r}}\big|\ln(|D\ue|^2+\epsilon)\big|dxdt 
    +C\int_{B_{2r}}\big|\ln(|D\ue(x,t_0)|^2+\epsilon)\big|dx\Big)
    \end{aligned}
\end{equation}
where $C=C(n,p,\gamma,s,\lambda,w_1,w_2,w_3,w_4)>0$.
\end{lemma}
The last two integrals on the right hand side of \eqref{eq:lemma-for-final-estimate} do not appear if $s\neq\gamma-p+2$. The source of such an error terms in the case $s=\gamma-p+2$ is the logarithm in Lemma \ref{lem:regularized-time-derivative-structure} when $\beta=-2$.
%--------------------------------------------------------------------
\section{Smooth case with non-zero gradient}\label{sec:smcase}
%--------------------------------------------------------------------
Let $1<p<\infty$ and $-1<\gamma<\infty$. In this section we assume that $u\colon\Om_T\to\R$ is a smooth solution to 
\begin{equation} \label{eq:PDE}
   u_t-|Du|^{\gamma}\big(\Delta u+(p-2)\ilN u\big)=0,
\end{equation}
such that $Du\neq 0$. That is, $u$ does not have critical points in space.
Our main result in this case is the following a priori estimate. Usually extending a regularity result to a general nonsmooth case is quite straightforward.
\begin{proposition} \label{prop:smooth}
Let $n\geq 2$, $1<p<\infty$ and $-1<\gamma<\infty$.
Let $u\colon\Om_T\to\R$ be a smooth solution to \eqref{eq:PDE} such that $Du\neq0$. If
\begin{align} \label{eq:Range-of-s}
    s>\max\Big\{-1-\frac{p-1}{n-1},\gamma+1-p\Big\},
\end{align}
then for any concentric parabolic cylinders $Q_r\subset Q_{2r}\Subset\Om_T$, we have the estimate
\begin{equation}
    \int_{Q_r}\abs{D(\abs{Du}^{\frac{p-2+s }{2}} Du)}^2dxdt
    \leq
    \frac{C}{r^2}\Big(
    \int_{Q_{2r}}|Du|^{p+s}dxdt
    +
    \int_{Q_{2r}}|Du|^{p+s-\gamma}dxdt
    \Big),
\end{equation}
where $C=C(n,p,\gamma,s)>0$.
\end{proposition}

The following Lemma, Lemma \ref{lem:smooth-key-lemma}, is the main ingredient in the proof of Proposition \ref{prop:smooth}. Thus we postpone the proof of Proposition \ref{prop:smooth} until after the proof of  Lemma \ref{lem:smooth-key-lemma}.

 In the following lemma we consider the weighted sum
\begin{equation} \label{eq:definition-of-S-smooth}
   S=w_1GD_1(p-2+s)+w_2GD_2(p-2+s)
\end{equation}
where $w_1,w_2\in\R$, and the notation was defined in \eqref{eq:S}.  
Note that since $\epsilon=0$ in this section, the terms with weights $w_3$ and $w_4$ in \eqref{eq:S} disappear.
The purpose of Lemma \ref{lem:smooth-key-lemma} is to show that under restriction \eqref{eq:Range-of-s}, we can find positive weights $w_1=w_1(n,p,\gamma,s)>0$ and $w_2=w_2(n,p,\gamma,s)>0$ such that $S$ has a suitably nonnegative lower bound to make Lemma \ref{lem:ad-hoc} applicable. Moreover, by the proof of Lemma \ref{lem:smooth-key-lemma} and Sylvester's condition, we can choose the value $c=c(n,p,\gamma,s)>0$ small enough such that for $M$ in the proof it holds
$$M_{11}\geq c
    \quad\text{and}\quad
    \det(M)\geq c.$$
The proof of Proposition \ref{prop:smooth} is then finished by using Lemma \ref{lem:lemma-for-final-estimate}.

\begin{lemma} \label{lem:smooth-key-lemma}
Let $n\geq 2$, $1<p<\infty$ and $-1<\gamma<\infty$.
Let $u\colon\Om_T\to\R$ be a smooth solution to \eqref{eq:PDE} such that $Du\neq0$. If \eqref{eq:Range-of-s} holds, then we can select $w_1=w_1(n,p,\gamma,s)>0$ and $w_2=w_2(n,p,\gamma,s)>0$, such that
\begin{align*}
    |Du|^{p-2+s}\Big\{c|D_T|Du||^2+Q\Big\}\le S
\end{align*}
where $c=c(n,p,\gamma,s)>0$ and
$$ Q=\la \bar{x},M\bar{x}\ra $$
with 
$\bar{x}=(\Delta_Tu,\ilN u)^T\in\R^2$ and 
a  uniformly bounded positive definite (with a uniform constant) symmetric matrix  $M=M(n,p,\gamma,s)\in \R^{2\times 2}$.
\end{lemma}

\begin{proof}
Similarly as in \eqref{eq:replace-time-derivatives}, recalling that $\eps=0$, by expressions \eqref{eq:c1}, we arrive at
\begin{equation} \label{eq:S-and-PDE}
\begin{aligned}
   |Du|^{p-2+s}&
    \Big\{
    w_1|D^2u|^2+w_1(p-2+s)|D_T|Du||^2+(w_2-w_1)(\Delta_T u)^2 \\
    &\quad
    +\big(w_2(p-1)(p-1+s-\gamma)-w_1\big)(\ilN u)^2 \\
    &\quad
    +\big(w_2(2p-2+s-\gamma)-w_1(p+s)\big)\Delta_Tu\ilN u
    \Big\}=S.
\end{aligned}
\end{equation}
We estimate $|D^2u|^2$ on the left hand side of \eqref{eq:S-and-PDE} from below by the fundamental inequality, Lemma \ref{lem:FundamentalInequality}. This yields the following lower bound for $S$
\begin{equation*}
\begin{aligned}
    &|Du|^{p-2+s}
    \Big\{
    w_1(p+s)|D_T|Du||^2
    +Q
    \Big\}\le S,
\end{aligned}
\end{equation*}
where
\begin{align*}
    Q:
    &=
    \Big(w_2- \frac{n-2}{n-1} w_1\Big)(\Delta_T u)^2
    +w_2(p-1)(p-1+s-\gamma)(\ilN u)^2 \\
    &\quad
    +\big(w_2(2p-2+s-\gamma)-w_1(p+s)\big)\Delta_Tu\ilN u.
\end{align*}
We write $Q$ more compactly as
$$ Q=\la\bar{x},M\bar{x}\ra, $$
where $\bar{x}=(\Delta_T u,\ilN u)^T\in \R^2$ is a vector and
$$ M:=
\begin{bmatrix}
w_2- \displaystyle{\frac{n-2}{n-1}} w_1 & \displaystyle{\frac{1}{2}}\big(w_2(2p-2+s-\gamma)-w_1(p+s)\big) \\
\displaystyle{\frac{1}{2}}\big(w_2(2p-2+s-\gamma)-w_1(p+s)\big) & w_2(p-1)(p-1+s-\gamma)
\end{bmatrix}
$$
is a symmetric $2\times2$-matrix.
We claim that under assumption \eqref{eq:Range-of-s} we can choose $w_1,w_2\in\R$ such that $M$ is uniformly bounded positive definite (with a uniform
constant).

If $n=2$, this is easy to see by selecting
$$ w_1=2p-2+s-\gamma \quad\text{and}\quad
w_2=p+s, $$
because then
$$ M=
\begin{bmatrix}
p+s & 0 \\
0 & (p+s)(p-1)(p-1+s-\gamma)
\end{bmatrix}
$$
and hence
$$ Q=(p+s)\left((\Delta_T u)^2+(p-1)(p-1+s-\gamma)(\ilN u)^2\right). $$
In other words, with such choice of $w_1$ and $w_2$, the mixed term $\Delta_T u\ilN u$ vanishes. Notice that \eqref{eq:Range-of-s} implies that $w_1>0$ and $w_2>0$.

For the higher dimensional case $n\geq3$, we set $w_1=1$ and  find $w_2=w_2(n,p,\gamma,s)>0$ such that $M$ is uniformly bounded positive definite (with a uniform constant). This is possible precisely when \eqref{eq:Range-of-s} holds: 
Since the proof is quite tedious, we postpone it to Lemma \ref{lem:Range-of-w2} in the appendix. 
\end{proof}

We are ready to give the proof of Proposition \ref{prop:smooth}.

\begin{proof}[Proof of Proposition \ref{prop:smooth}]
Let us fix $w_1=w_1(n,p,\gamma,s)>0$ and $w_2=w_2(n,p,\gamma,s)>0$ according to Lemma \ref{lem:smooth-key-lemma}. 
Lemma \ref{lem:ad-hoc} is then applicable because $w_1>0,\ w_3=0$ implies that $c_1=w_1+w_3 \kappa>0$ 
and the conclusion of Lemma \ref{lem:smooth-key-lemma} implies that \eqref{eq:lower-bound-S} holds.
Therefore, by Lemma \ref{lem:ad-hoc} there exists $\lambda=\lambda(n,p,\gamma,s,w_1,w_2,w_3,w_4)>0$
such that
$$ \lambda|Du|^{p-2+s}|D^2u|^2 \le S$$
in $\Om_T$. 
Now the desired estimate follows from Lemma \ref{lem:lemma-for-final-estimate}.
\end{proof}
%--------------------------------------------------------------------
Range (\ref{eq:Range-of-s}) in Proposition \ref{prop:smooth}, is optimal in the following sense: In the elliptic case, \cite{dongpzz20} and \cite{sarsa20}, the best known range is $s>-1-\frac{p-1}{n-1}$.
On the other hand, Example \ref{counterexample} below shows that in the parabolic case we cannot hope to reach any better range than $s>\gamma+1-p$. 
A counterexample of this type was used in \cite[Section 1.3]{dongpzz20} for the standard $p$-parabolic equation. 
%--------------------------------------------------------------------
\begin{example}[Counterexample]\label{counterexample}
Let $u\colon\Rn\times(0,\infty)\to\R$ be given by
$$ u(x,t):=Ct+|x_1|^{\alpha} $$
for some $C\in\R$ and $\alpha>0$.
Note that
$$ |Du|^{\gamma}\plN u=\alpha^{\gamma+1}(\alpha-1)(p-1)|x_1|^{(\alpha-1)(\gamma+1)-1}. $$
Hence, if
$$ \alpha=1+\frac{1}{\gamma+1}
\quad\text{and}\quad C=\alpha^{\gamma+1}(\alpha-1)(p-1) $$
then $u$ solves \eqref{eq:GeParabEq} in the classical sense whenever $x_1\neq0$. 
Indeed, by a direct computation, we have
$$u_{x_1}=\alpha|x_1|^{\alpha-2}x_1, \quad u_{x_i}=0 \quad {\rm for}\quad i=2,\cdots,n,$$
and
$$u_{x_1x_1}=\alpha(\alpha-1)|x_1|^{\alpha-2}, \quad u_{x_ix_j}=0,$$
where $i,j=1,\cdots,n$ and $i$ and $j$ are not both 1.

Next we verify that the function $u$ is a viscosity solution in the whole $\Rn$ according to Definition \ref{def:visc} also at those points where $x_1=0$. Whenever $x_1\neq 0$, $x_0=(x_1,\ldots,x_n)$, and the test function $\varphi$ touches $u$ at $(x_0,t_0)$ from below (the argument is analogous from above), we may use the facts that 
\begin{align*}
    D\varphi (x_0,t_0)=Du(x_0,t_0)\neq 0,\quad 
    \phi_t(x_0,t_0)=u_t(x_0,t_0)=0 
\end{align*}
and
$$D^2\varphi (x_0,t_0)\le D^2u(x_0,t_0).$$

Let us consider the points where $x_1=0$. We study the degenerate case $\gamma>0$ and the singular case $-1<\gamma\leq 0$ separately.
If $x_1=0$ and $\gamma>0$, then there are no test functions touching $u$ from above and for a test function $\varphi$ touching from below, we have $D\varphi (x_0,t_0)=Du(x_0,t_0)=0$  and 
$$
\varphi_t(x_0,t_0)=u_t(x_0,t_0)=C.
$$ 
Since $C>0$, the function $u$ is a viscosity supersolution. 

The given function is also a viscosity solution whenever $-1<\gamma\le 0$: the proof for the supersolution property is the same as in the degenerate case above. It is also a subsolution because  (similarly to the degenerate case) there are no admissible test functions touching $u$ from above.
We provide a detailed proof of this fact. Thriving for a contradiction, suppose that there is an admissible test function $\vp$ touching $u$ at $(x_0,t_0)$ with $x_0=(0,\ldots,0)$ (for simplicity) from above. Then necessarily
\begin{align*}
    \vp_t(x_0,t_0)=C>0.
\end{align*}
 By the definition of a viscosity solution it holds that 
\begin{align*}
   \phi(x,t)=u(x_0,t_0)+\vp_t(x_0,t_0)(t-t_0)+f(\abs x)+\sigma(t-t_0)
\end{align*}
is an admissible test function touching strictly from above. By strict touching and regularity of $u$, by translating with respect to $x_1$ and lifting we may assume that $\phi$ touches $u$ at a point $(x,t_0)$, $x=(\eps,0,\ldots,0)$, with small $\eps>0$. Also observe that by an approximation, we could assume that $\sigma$ is a $C^2$ function, but we omit this step as well.   Also recall the notation $g(y)=f(\abs y)$ and that 
\begin{align}
    \lim_{y\to 0,y\neq 0}F\big(Dg(y),D^2g(y)\big)=0.
\end{align}
Then by this and the counter assumption it holds at a point $(x,t_0)$ for $x$ close enough $x_0$ that
\begin{align} \label{eq:counterexample-1}
    \phi_t(x,t_0)-F\big(D\phi(x,t_0),D^2\phi(x,t_0)\big)=\vp_t(x_0,t_0)-F\big(Dg(x),D^2g(x)\big)>0.
\end{align}
On the other hand, since $u$ is now $C^2$-function with the explicit formula, we have 
\begin{equation} \label{eq:counterexample-2}
\begin{aligned}
    \phi_t(x,t_0)-F\big(D\phi(x,t_0),D^2\phi(x,t_0)\big)
    &=\vp_t(x_0,t_0)-F\big(Dg(x),D^2g(x)\big) \\
    &=u_t(x,t_0)-F\big(Dg(x),D^2g(x)\big) \\
    &\le u_t(x,t_0)-F\big(Du(x,t_0),D^2u(x,t_0)\big)=0,
\end{aligned}
\end{equation}
which contradicts inequality \eqref{eq:counterexample-1}.

In the above inequality we used the fact that since $\phi$ touches $u$ from above at $(x,t_0)$ we have $D^2g(x)\ge D^2 u(x,t_0)$ and $Dg(x)=Du(x,t_0)\neq 0$ and thus
\begin{align*}
   F\big(Dg(x),D^2g(x)\big)\ge  F\big(Du(x,t_0),D^2u(x,t_0)\big).
\end{align*}

We study the local $W^{1,2}$-regularity of $|Du|^{\frac{p-2+s}2}Du$ for $s\in\R$ and see what kind of restrictions for $s$ arise. We have
\begin{align*}
 \left|D(|Du|^{\frac{p-2+s}2}Du)\right|
&=\frac{1}{2}\alpha^{\frac{p+s}{2}}(\alpha-1)(p+s)|x_1|^{\frac{(\alpha-1)(p+s)}{2}-1}\\
&=C(p,s,\gamma)|x_1|^{\frac{ p+s }{2(\gamma+1)}-1}.
\end{align*}
The function $D(|Du|^{\frac{p-2+s}2}Du)$ locally belongs to $L^2(\Rn\times(0,\infty))$ if and only if
$$ 2\Big(\frac{ p+s }{2(\gamma+1)}-1\Big)>-1, $$
that is,
$$ s>\gamma+1-p. $$
Observe that range condition (\ref{eq:Range-of-s}) gives this in the plane, but in higher dimensions we have an additional restriction, which is the same restriction as in the elliptic case.

When $s=2-p$, then for $W^{2,2}$-regularity, the range $$-1<\gamma<1$$ is sharp in the plane.
\end{example}
%--------------------------------------------------------------------
\begin{remark}
Also the case $n=1$ holds. Recall that the key point is identity \eqref{eq:RoadmapBasicIdentity}, that is,
\begin{equation}
\begin{aligned}
    &(|D\ue|^2+\epsilon)^{\frac{p-2+s}{2}}
    \Big\{|D^2\ue|^2+(p-2+s)\frac{|D^2\ue D\ue|^2}{|D\ue|^2+\epsilon}
    +s(p-2)\frac{(\il \ue)^2}{(|D\ue|^2+\epsilon)^2} \\
    &\quad
    +(p-2-\gamma)(|D\ue|^2+\epsilon)^{- {\gamma}/{2}}\ue_t\frac{\il \ue}{|D\ue|^2+\epsilon} \Big\} \\
    &=
    \diverg\big((|D\ue|^2+\epsilon)^{\frac{p-2+s}{2}}AD^2\ue D\ue \big)
    -\frac{\big((|D\ue |^2+\epsilon)^{\frac{p+s-\gamma}{2}}\big)_t}{p+s-\gamma},
\end{aligned}
\end{equation}
provided that $s\neq\gamma-p$. If $n=1$ this reduces to
\begin{equation} \label{eq:one-dim-BasicIdentity}
\begin{aligned}
    &(|D\ue|^2+\epsilon)^{\frac{p-2+s}{2}}
    \Big\{1+(p-2+s)\theta+s(p-2)\theta^2 \\
    &\quad
    +(p-2-\gamma)\big((p-2)\theta+1\big)\theta\Big\}|D^2\ue|^2 \\
    &=
    \diverg\big((|D\ue|^2+\epsilon)^{\frac{p-2+s}{2}}(D^2\ue D\ue-\Delta\ue D\ue) \big) \\
    &\quad
    +\ue_t\diverg\big((|D\ue|^2+\epsilon)^{\frac{p-2+s-\gamma}{2}}D\ue\big).
\end{aligned}
\end{equation}
The left hand side of \eqref{eq:one-dim-BasicIdentity} is
\begin{align*}
    &(|D\ue|^2+\epsilon)^{\frac{p-2+s}{2}}
    \Big\{(p-1)(p-1+s-\gamma)\theta^2
    +(2p-2 +s-\gamma)\theta\kappa +\kappa^2\Big\}|D^2\ue|^2 \\
    &\quad
    \geq \lambda (|D\ue|^2+\epsilon)^{\frac{p-2+s}{2}}|D^2\ue|^2,
\end{align*}
for some constant $\lambda=\lambda(p,\gamma,s)>0$, provided that $s>\gamma+1-p$. From this it is easy to derive the desired integral estimate. We conclude that Proposition \ref{prop:smooth} holds in case $n=1$ without the additional smoothness assumptions for $u$, and with the interpretation
$$ s>\max\Big\{-1-\frac{p-1}{n-1},\gamma+1-p\Big\}=\max\{-\infty,\gamma+1-p\}=\gamma+1-p. $$
\end{remark}
%--------------------------------------------------------------------
\section{Removing the smoothness assumption}
\label{sec:full}
%--------------------------------------------------------------------
Section \ref{sec:smcase} gives a formal derivation of the regularity estimate under the assumption that the gradient of a solution does not vanish. In this section, we remove the additional assumption in a certain range of parameters by regularizing the equation and then finally pass to a limit to obtain the result 
for the original equation.
%--------------------------------------------------------------------
\subsection{Regularization}\label{sec:reg-pro}
%--------------------------------------------------------------------
 Let $\ue\colon \Omega_T\rightarrow\R$ be a smooth solution to the equation 
\begin{equation} \label{eq:PDEwithGeneralDegeneracy}
    \ue_t-(|D\ue|^2+\epsilon)^{\gamma/2}\Big(\Delta\ue+(p-2)\frac{\il\ue}{|D\ue|^2+\epsilon}\Big)=0
\end{equation}
where $1<p<\infty$, $-1<\gamma<\infty$, and $\ez>0$ is a regularization parameter.
As explained in Section \ref{sec:eps-plan}, the mismatch between the second order differential quantities in the fundamental inequality and the regularized equation and consequently in the basic estimate causes that some of the formal calculations do not work as such even if most of the steps work for general $s$. In particular positive definiteness of the quadratic form becomes an issue. 

 In this section, partly for the convenience of the reader, we have decided to limit ourselves to the planar case $n=2$ and focus on the square-integrability of the second order derivatives $D^2u$, that is, we consider the case $s=2-p$.  In this case the range condition in \eqref{range:expected-2} that is
 \begin{align*}
      s>\max\Big\{-1-\frac{p-1}{n-1},\gamma+1-p\Big\}
 \end{align*}
 reduces to 
\begin{equation*}
    1<p<\infty \quad\text{and}\quad -1<\gamma<1.
\end{equation*}
Then range (i) 
\begin{equation} \label{eq:Roadmap-range-i}
    1<p\leq5 \quad\text{and}\quad -1<\gamma<1,
\end{equation} 
in Theorem \ref{resultforwholep} and the Proposition below
will be obtained by a straightforward generalization of the formal calculation ($\epsilon=0$). 
%--------------------------------------------------------------------
That is, we consider the sum
\begin{align} \label{eq:Roadmap-S-plane}
    S=
    w_1\diverg (D^2\ue D\ue-\Delta \ue D\ue)
    +w_2\ue_t\diverg\big((|D\ue|^2+\ez)^{-\frac{\gamma}{2}}D\ue\big),
\end{align}
and show that if \eqref{eq:Roadmap-range-i} holds, then we can find $w_1,w_2>0$ such that
\begin{align} \label{eq:Roadmap-lower-bound-S-regularized}
c \abs{D_T\abs{D\ue}}^2 +Q\le S,
\end{align}
where $c>0$ and $Q$ is positive definite. For range (ii) in the Proposition below which is the same as in Theorem \ref{resultforwholep}, we instead consider the full $S$ as defined in \eqref{eq:Roadmap-S-with-errorterms-alpha-is-0} or equivalently in \eqref{eq:S}.

Our main result for $\ue$ is the following.

\begin{proposition} \label{prop:Regularized-resultforwholep}
Let $n=2$. Let $\ue\colon \Omega_T\to\R$ be a smooth solution to \eqref{eq:PDEwithGeneralDegeneracy}.
If $p$ and $\gamma$ satisfy one of the following conditions:
\begin{enumerate}
    \item[(i)] $1<p\leq5$ and $-1<\gamma<1$; or
    \item[(ii)] $1<p<\infty$ and $-1<\gamma<\sqrt{2}-\frac12$,
\end{enumerate}
then for any concentric parabolic cylinders $Q_r\subset Q_{2r}\Subset \Omega_T$ with center point $(x_0,t_0)\in\Om_T$, we have the estimate
\begin{equation}
    \begin{aligned}
    \int_{Q_r}|D^2\ue|^2dxdt
    &\leq
    \frac{C}{r^2}\Big(
    \int_{Q_{2r}}|D\ue|^2dxdt 
    +\int_{Q_{2r}}(|D\ue|^2+\epsilon)^{\frac{2-\gamma}{2}}dxdt\Big) \\
    &\quad
    +\epsilon\Big(\frac{C}{r^2}\int_{Q_{2r}}\big|\ln(|D\ue|^2+\epsilon)\big|dxdt
    +C\int_{B_{2r}}\big|\ln(|D\ue(x,t_0)|^2+\epsilon)\big|dx
    \Big)
    \end{aligned}    
\end{equation}
where $C=C(p,\gamma)>0$.
\end{proposition}

The proof of Proposition \ref{prop:Regularized-resultforwholep} is postponed to the end of the section. 
The main ingredients of the proof of Proposition \ref{prop:Regularized-resultforwholep} are the following lemmas, Lemma \ref{lem:simple} and Lemma \ref{lem:vanish}. The first lemma, Lemma \ref{lem:simple}, yields case (i). The second lemma, Lemma \ref{lem:vanish} yields case (ii).
In both lemmas we consider the same weighted sum as before now selecting $s=2-p$ i.e.\
   \begin{align}\label{eq:S-special}
    S&=w_1GD_1(p-2+s)+w_2GD_2(p-2+s)+\epsilon w_3GD_1(p-4+s)+\epsilon w_4GD_2(p-4+s)\nonumber\\
    &=w_1GD_1(0)+w_2GD_2(0)+w_3\epsilon GD_1(-2)+w_4\epsilon GD_2(-2),
\end{align} 
where $w_1,w_2,w_3,w_4\in\R$ are some weights, and the notation was defined in \eqref{eq:S}. 

The purpose of Lemma \ref{lem:simple} and Lemma \ref{lem:vanish} is to show that under restrictions (i) and (ii), respectively, we can find suitable weights $w_1,w_2,w_3$ and $w_4$, that only depend on $p$ and $\gamma$, such that $S$ has a suitable lower bound.

\begin{lemma} \label{lem:simple}
Let $n=2$, $S$ be as in \eqref{eq:S-special}, and (i) in Proposition \ref{prop:Regularized-resultforwholep} hold. For $\eta=\eta(p,\gamma)>0$ small enough, if 
\begin{equation} \label{eq:simple-weights}
\begin{aligned}
\begin{cases}
    w_1= p-\gamma -2\sqrt{(p-1)(1-\gamma)}  +\eta, \quad    &w_2=2, \\
    w_3=0, \quad &w_4=0,
\end{cases}
\end{aligned}
\end{equation}
then 
\begin{align} \label{eq:simple-lower-bound-for-S}
    S\geq c|D_T|D\ue||^2+Q
\end{align}
where $c=c(p,\gamma)>0$ and
$$ Q=\la \bar{x},M\bar{x}\ra $$
with $\bar{x}=(\Delta_T\ue,\ilN \ue)^T\in\R^2$ and a  uniformly bounded positive definite (with a uniform constant)  symmetric matrix $M=M(p,\gamma)\in \R^{2\times 2}$.
\end{lemma}

\begin{lemma} \label{lem:vanish}
Let $n=2$, $S$ be as in \eqref{eq:S-special},
and (ii) in Proposition \ref{prop:Regularized-resultforwholep} hold. If 
\begin{equation} \label{eq:vanish-weights}
\begin{aligned}
\begin{cases}
    w_1=p-\gamma,\quad 
    &w_2=2 \\
    w_3=4-p+\gamma,\quad &w_4=2,
\end{cases}
\end{aligned}
\end{equation}
then a statement similar to that in Lemma \ref{lem:simple} holds.
\end{lemma}

To begin with, recall from 
\eqref{eq:replace-time-derivatives} that $S$ can be written as
\begin{align}\label{eq:regularid-decomposition}
    S
    &=
    c_1|D^2\ue|^2+c_2|D_T|D\ue||^2+(c_3-c_1)(\Delta_T \ue)^2+\big((c_3+c_4)P_\theta-c_1\big) (\ilN \ue)^2\nonumber\\
    &\quad
    +\big(c_3P_\theta+(c_3+c_4)-(2c_1+c_2)\big)\Delta_T \ue \ilN \ue.
\end{align}	
where
\begin{align}\label{cdt:ci}
\begin{cases}
    c_1=w_1+w_3\kappa, \quad
    &c_2=-2w_3\theta\kappa, \\
    c_3=w_2+w_4\kappa,\quad 
    &c_4=-(w_2\gamma+w_4\kappa(2+\gamma))\theta,
\end{cases}
\end{align}
and
$$P_\theta=(p-2)\theta+1\in(0,\infty),
\quad 
\theta=\frac{|Du|^2}{|Du|^2+\epsilon}\in[0, 1),
\quad 
\kappa=1-\theta=\frac{\epsilon}{|D\ue|^2+\epsilon}\in(0,1].
$$
Fundamental equality \eqref{eq:tri-ineq} in the plane yields that
\begin{equation}\label{eq:regularid-fun-ineq}
\begin{aligned}
S=&c_1\big(2|D_T|D\ue| |^2
    + (\Delta_T \ue)^2  
    +(\ilN \ue )^2\big)+c_2|D_T|D\ue||^2+(c_3-c_1)(\Delta_T \ue)^2\\
    &+\big((c_3+c_4)P_\theta-c_1\big) (\ilN \ue)^2
    +\big(c_3P_\theta+(c_3+c_4)-(2c_1+c_2)\big)\Delta_T \ue \ilN \ue\\
=&(2c_1 +c_2)|D_T|D\ue||^2+Q.
\end{aligned} 
\end{equation}
where
\begin{equation}\label{eq:Q-full}
 \begin{aligned}
Q=c_3  (\Delta_T \ue )^2+ (c_3+c_4)P_\theta (\ilN \ue)^2 +\big(c_3P_\theta+(c_3+c_4)-(2c_1+c_2)\big) \Delta_T \ue  \ilN \ue,
\end{aligned}   
\end{equation}
is a quadratic form in $ \Delta_T \ue $ and $ \ilN \ue $.
We write $Q$ compactly as
$$ Q=\la\bar{x},M\bar{x}\ra, $$
where $\bar{x}=(\Delta_T \ue,\ilN \ue)^T\in \R^2$ is a vector and
$$ M:=
\begin{bmatrix}
c_3 & \displaystyle{\frac{1}{2}\big(c_3P_\theta+(c_3+c_4)-(2c_1+c_2)\big)} \\
\displaystyle{\frac{1}{2}\big(c_3P_\theta+(c_3+c_4)-(2c_1+c_2)\big)} & (c_3+c_4)P_\theta
\end{bmatrix}
$$
is a symmetric $2\times2$-matrix. 

To prove Lemma \ref{lem:simple} and \ref{lem:vanish}, it now suffices to check that \eqref{eq:regularid-fun-ineq} satisfies all the requirements of the lemmas:
The coefficient of $|D_T|D\ue||^2$ in \eqref{eq:regularid-fun-ineq} needs to be bounded from below by a positive constant, that is,
\begin{align}\label{cdt:2c1+c2}
2c_1+c_2=2(w_1+w_3\kappa)-2w_3\theta\kappa\geq c
\end{align}
uniformly in $\Omega_T$.
For the quadratic form $Q$, we need to analyse the uniform boundedness and uniform positive definiteness of matrix $M$.
Uniform boundedness is quite straightforward, so we focus our attention on the uniform positive definiteness. By Sylvester's condition, it suffices to check that
\begin{align}\label{cdt:c3}
c_3=w_2+w_4\kappa\geq c,
\end{align}
and
\begin{align}\label{cdt:detq}
\det(M)=c_3(c_3+c_4)P_\theta-\frac{\big(c_3P_\theta+(c_3+c_4)-(2c_1+c_2)\big)^2}{4}\geq c
\end{align}
uniformly in $\Omega_T$.
%---------------------------------------------------------------------
Next we prove Lemma \ref{lem:simple}, which implies nonnegativity of the necessary terms when $1<p\leq5$ and $-1<\gamma<1$. In this case a simple choice  of the weights $w_3=w_4=0$ will work. 
\begin{proof}[Proof of Lemma \ref{lem:simple}]
Similarly to the smooth case, we start with $w_3=w_4=0$, plug these values into \eqref{cdt:ci}, and obtain
\begin{align*}
\begin{cases}
    c_1=w_1,\quad  &c_2=0,\\
    c_3=w_2, \quad 
    &c_4=-w_2\gamma\theta.
\end{cases}
\end{align*}
This together with \eqref{eq:regularid-fun-ineq} gives
\begin{equation}
    \begin{aligned} \label{eq:simple-regularid-fun-ineq}
S=2w_1|D_T|D\ue||^2+ w_2(\Delta_T \ue )^2+w_2 P_\theta R_\theta(\ilN \ue)^2+\big(w_2(P_\theta+R_\theta)-2w_1\big)\Delta_T \ue \ilN \ue,
\end{aligned}
\end{equation}
where we denote $R_\theta:=1-\gamma\theta\in(0,2),$
for the sake of brevity. To simplify the above identity, we select $w_2=2.$ Thus
 \begin{equation}
    \begin{aligned}  
S=&2w_1|D_T|D\ue||^2+ 2\Big((\Delta_T \ue )^2+  P_\theta R_\theta(\ilN \ue)^2+( P_\theta+R_\theta -w_1)\Delta_T \ue \ilN \ue\Big)\\
=& 2w_1|D_T|D\ue||^2+Q,
\end{aligned}
\end{equation}
where the matrix of the quadratic form $Q$ is
$$ M(\theta):=
\begin{bmatrix}
2 &  P_\theta+R_\theta-w_1 \\
P_\theta+R_\theta-w_1 & 2P_\theta R_\theta
\end{bmatrix}.
$$
The determinant of $M(\theta)$ is uniformly positive if and only if
$$P_\theta R_\theta-\frac{(P_\theta+R_\theta -w_1)^2}{4}\geq c>0, $$
that is,
 $$ X_2(\theta):= (\sqrt{P_\theta} +\sqrt{R_\theta})^2 >w_1>  (\sqrt{P_\theta} -\sqrt{R_\theta})^2:=X_1(\theta)$$
 uniformly in $\Omega_T$.
Thus it suffices to verify
$$\inf_\theta X_2(\theta)>\sup_\theta X_1(\theta).$$
Computing the derivative of $X_1$ with respect to $\theta$, one has
\begin{align*}
   X_1^\prime(\theta)
   =&p-2-\gamma-\frac{(p-2) R_\theta-\gamma P_\theta }{\sqrt{P_\theta  R_\theta}}\\
   =&\frac{(\sqrt{P_\theta}-\sqrt{R_\theta})\big((p-2)\sqrt{R_\theta}+\gamma\sqrt{P_\theta}\big)}{ \sqrt{P_\theta R_\theta }}.
\end{align*} 
Then in the eligible range of parameters 
$X_1^\prime(\theta)=0$ if and only if
$p-2+\gamma=0.$ 
Hence, by considering the values at the endpoints, we obtain the supremum of $X_1$ with respect to $\theta$:
\begin{align*}
\sup_\theta X_1 (\theta)&
=\max\{X_1 (0), X_1 (1)\}\\
&=\max\Big\{0, \big(\sqrt{p-1}-\sqrt{1-\gamma}\big)^2  \Big\}\\
&= \big(\sqrt{p-1}-\sqrt{1-\gamma}\big)^2.
\end{align*}

Similarly, we obtain the derivative of $X_2$ with respect of $\theta$:
$$X_2^\prime(\theta)=\frac{(\sqrt{P_\theta}+\sqrt{R_\theta})\big((p-2)\sqrt{R_\theta}-\gamma\sqrt{P_\theta}\big)}{ \sqrt{P_\theta R_\theta }},$$
and thus the eligible stationary point is
$$\theta_2=\frac{p-2-\gamma}{(p-2)\gamma}$$
if $(p-2)\gamma>0$.
Then
\begin{align*}
\inf_\theta X_2(\theta)&
=\min\Big\{X_2(0),X_2\Big(\frac{p-2-\gamma}{(p-2)\gamma}\Big), X_2(1)\Big\}\\
&=\min\Big\{4,  \frac{p-2}{\gamma}+\frac{\gamma}{p-2}+2,\big(\sqrt{p-1}+\sqrt{1-\gamma}\big)^2 \Big\}\\
&=\min\Big\{4,  \big(\sqrt{p-1}+\sqrt{1-\gamma}\big)^2 \Big\}.
\end{align*}
Obviously, we have $$\big(\sqrt{p-1}+\sqrt{1-\gamma}\big)^2>\big(\sqrt{p-1}-\sqrt{1-\gamma}\big)^2.$$
Note that
$$4>\big(\sqrt{p-1}-\sqrt{1-\gamma}\big)^2$$
is equivalent to                               
$$1<p\leq 5 \quad \text{and}\quad -1<\gamma<1,$$
or $5<p<7+4\sqrt{2}$ and $-1<\gamma<-2-p+4\sqrt{p-1}$.  
Thus if $p$ and $\gamma$ satisfy range (i), for small enough $\eta=\eta(p,\gamma)>0$, 
in addition to the above choice $w_2=2$, we set
\begin{align*}
     w_1=\big(\sqrt{p-1}-\sqrt{1-\gamma}\big)^2+\eta.
\end{align*}
The proof is finished.
\end{proof}
%---------------------------------------------------------------------
Next  we prove Lemma \ref{lem:vanish}, which implies nonnegativity of the necessary terms when  $1<p<\infty$ and $-1<\gamma<\sqrt{2}-\frac12$. In this case, we use a choice  of the weights which leads to the vanishing coefficient of the mixed term $\Delta_T \ue \ilN \ue $ in \eqref{eq:Q-full}.

To be more precise, at the beginning of this section, we obtained three conditions \eqref{cdt:2c1+c2}, \eqref{cdt:c3} and \eqref{cdt:detq}, i.e.\ that 
\begin{align}\label{cdt:vanish-all}
\begin{cases}
    2c_1+c_2&=2(w_1+w_3\kappa)-2w_3\theta\kappa\geq c,\\
    c_3&=w_2+w_4\kappa\geq c,\\
    \det(M)&=c_3(c_3+c_4)P_\theta-\frac{\left(c_3P_\theta+(c_3+c_4)-(2c_1+c_2)\right)^2}{4}\geq c
\end{cases}
\end{align}
need to hold uniformly in $\Omega_T$. Here 
$$ M=
\begin{bmatrix}
c_3 & \displaystyle{\frac{1}{2}\big(c_3P_\theta+(c_3+c_4)-(2c_1+c_2)\big)} \\
\displaystyle{\frac{1}{2}\big(c_3P_\theta+(c_3+c_4)-(2c_1+c_2)\big)} & (c_3+c_4)P_\theta
\end{bmatrix}
$$
is the coefficient matrix 
of quadratic form
\begin{align*}
Q=c_3  (\Delta_T \ue )^2+ (c_3+c_4)P_\theta (\ilN \ue)^2 +\big(c_3P_\theta+(c_3+c_4)-(2c_1+c_2)\big) \Delta_T \ue  \ilN \ue.
\end{align*}
To simplify the computations in checking the last condition in (\ref{cdt:vanish-all}), we will consider a special case where the coefficient $c_3P_\theta+(c_3+c_4)-(2c_1+c_2)$ of the mixed term $\Delta_T \ue  \ilN \ue $ vanishes.

\begin{lemma}[Vanishing mixed term]\label{lem:vanishmixt}
The mixed term $ \Delta_T \ue  \ilN \ue $ in $Q$ vanishes, i.e.
$$ c_3P_\theta+(c_3+c_4)-(2c_1+c_2)=0 $$
uniformly in $\Omega_T$ if and only if 
$$
\begin{cases}
    2w_1=(p-\gamma)w_2, \\
    2w_3=(4-p+\gamma)w_4, \\
    (p-2-\gamma)(w_4-w_2)=0.
\end{cases}
$$
\end{lemma}

\begin{proof}
Recall that $\theta=1-\kappa$, $\kappa>0$, and $P_\theta=(p-2)\theta+1$. Then recalling the expressions of $c_1,\cdots,c_4$ in \eqref{cdt:ci},  we can write the coefficient of the mixed term  $ \Delta_T \ue  \ilN \ue $ as a polynomial of $\kappa$ as 
\begin{align*}
&c_3P_\theta+(c_3+c_4)-(2c_1+c_2)\\
&\hspace{1 em}=\big((4-p+\gamma)w_4-2w_3\big)\kappa^2+(p-2-\gamma)(w_4-w_2)\kappa+(p-\gamma)w_2-2w_1.
\end{align*}
Set all the coefficients to be zero, we have the desired condition.
\end{proof}

By the above Lemma, we can easily to obtain the following result.

\begin{corollary}\label{cor:wi-vanish}
If
\begin{align}\label{cdt:wi-vanish}
\begin{cases}
    w_1 =p-\gamma,\quad 
    &w_2=2, \\
    w_3 =4-p+\gamma,\quad  
    &w_4=2,
\end{cases}
\end{align}
then the mixed term $ \Delta_T \ue  \ilN \ue $ in $Q$ vanishes.
\end{corollary}
The above corollary gives a choice of the coefficients $w_1,w_2,w_3$ and $w_4$ to obtain the vanishing coefficient of mixed term $ \Delta_T \ue  \ilN \ue $. This then helps us in proving Lemma \ref{lem:vanish}.

\begin{proof}[Proof of Lemma \ref{lem:vanish}]
If $w_1,w_2,w_3$ and $w_4$ satisfy \eqref{cdt:wi-vanish}, then by Corollary \ref{cor:wi-vanish}, the last condition in  \eqref{cdt:vanish-all} reduces to checking that
$$\det(M)=c_3(c_3+c_4)P_\theta \geq c.$$
Since $$P_\theta=(p-2)\theta+1\geq \min\{p-1,1\}>0,$$ 
sufficient conditions to obtain \eqref{cdt:vanish-all} can be written as
\begin{align}\label{cdt:vanish}
\begin{cases}
    2c_1+c_2&=2(w_1+w_3\kappa)-2w_3\theta\kappa\geq c,\\
    c_3&=w_2+w_4\kappa\geq c,\\
    c_3+c_4&=w_2+w_4\kappa-\big(w_2\gamma+w_4\kappa(2+\gamma)\big)\theta\geq c
\end{cases}
\end{align}
uniformly in $\Om_T.$

First, using  values \eqref{cdt:wi-vanish} in the first condition and replacing $\theta$ by $1-\kappa$, we have
\begin{align*}
2c_1+c_2
=&2(p-\gamma)+2(4-p+\gamma)\kappa^2.
\end{align*}
Since $\kappa$ is positive, the sign of the derivative with respect to $\kappa$ that is  $4(4-p+\gamma)\kappa$  is fixed. Then $2c_1+c_2$ with respect to $\kappa$ is monotone and the minimum point corresponds either $\kappa=0$ or $\kappa=1$.
Thus
\begin{equation} \label{eq:bound2c1+c2}
\begin{aligned}
2c_1+c_2\geq\min\{2(p-\gamma), 8\}>0.
\end{aligned}
\end{equation}
For the second condition, when $w_2=w_4=2$, it is obvious that
\begin{align}\label{eq:vnsc_3dt}
c_3=2+2\kappa\geq2>0.
\end{align}
%---------------------------------------------------
Finally, for the last condition plugging values \eqref{cdt:wi-vanish} in and rewriting as 
\begin{align*}
c_3+c_4
=&2(1-\gamma)-2\kappa+2(2+\gamma)\kappa^2.
\end{align*}
When the derivative of $c_3+c_4$ with respect to $\kappa$ vanishes, that is, 
$-2+4(2+\gamma)\kappa=0,$
one has
$$\kappa_1=\frac{1}{2(2+\gamma)}\in (0,1].$$
Then the minimum point is one of the boundary points or the extreme point $\kappa_1 $. 
Selecting $\kappa=\kappa_1$, we have
\begin{align}
    \label{eq:c3+c4}
c_3+c_4=2(1-\gamma)-\frac{1}{2(2+\gamma)}>0
\end{align}
if and only if
$$-1<\gamma<\sqrt{2}-\frac12.$$
If $\kappa=0$, we have $c_3+c_4=2(1-\gamma)$, and if $\kappa=1$, then $c_3+c_4=4$. It follows that the minimum is given by strictly positive expression \eqref{eq:c3+c4}, and the proof is finished.
\end{proof}

The proof of Proposition \ref{prop:Regularized-resultforwholep} now immediately follows.

\begin{proof}[Proof of Proposition \ref{prop:Regularized-resultforwholep}]
The result immediately follows from the previous lemmas, since under assumption (i),  Lemma \ref{lem:simple} implies that \eqref{eq:lower-bound-S} holds and thus
Lemma \ref{lem:ad-hoc} is applicable. 
Similarly under assumption (ii), Lemma \ref{lem:vanish} implies that
Lemma \ref{lem:ad-hoc} is applicable.
Now the desired estimate follows from Lemma \ref{lem:lemma-for-final-estimate}.
\end{proof}
%--------------------------------------------------------------------
\subsection{Passing to the original equation}\label{sec:ProofofMain}
%--------------------------------------------------------------------
In this section we justify the limiting argument to let $\epsilon\to0$ in Proposition \ref{prop:Regularized-resultforwholep} and thus derive our main result, Theorem \ref{resultforwholep}.

\begin{proof}[Proof of Theorem \ref{resultforwholep}]
Let $u\colon\Omega_T\rightarrow\R$ be a viscosity solution to 
\begin{align}\label{eq:GeParabEq-2}
    u_t-|Du|^{\gamma}\big(\Delta u+(p-2)\Delta_\infty^Nu\big)=0.
\end{align}
Let us fix concentric parabolic cylinders $Q_r\subset Q_{2r}\Subset\Om_T$
 with center point $(x_0,t_0)\in\Om_T$
and moreover, let us fix a smooth subdomain $U\Subset \Omega$ and $0<t_1<t_2<T$ such that $Q_{2r}\Subset U_{t_1,t_2}\Subset\Omega_T$.
For $\epsilon>0$ small, let us consider the Dirichlet problem 
\begin{equation}
\begin{cases}
\begin{aligned}
   \ue_t-(|D\ue|^2+\epsilon)^{\gamma/2}\Big(\Delta\ue+(p-2)\frac{\il\ue}{|D\ue|^2+\epsilon}\Big)=0
   \quad&\text{in } U_{t_1,t_2};\\
   \ue=u
   \quad&\text{on } \partial_pU_{t_1,t_2},
\end{aligned}
\end{cases}
\end{equation}
where
$$ \partial_pU_{t_1,t_2}:=(\overline U\times \{t_1\})\cup(
\partial U\times (t_1,t_2]) $$
is the parabolic boundary of $U_{t_1,t_2}$.
By the classical theory of uniformly parabolic equations, the above problem has a unique solution $\ue\in C^\infty( U_{t_1,t_2})\cap C( \overline{U} _{t_1,t_2})$.

Proposition \ref{prop:Regularized-resultforwholep} is applicable to $\ue$ and we conclude that
\begin{equation}\label{eq:RegularizedQuantitativeBoundforSecondDerivativesinBall}
    \begin{aligned}
    \int_{Q_r}|D^2\ue|^2dxdt
    &\leq
    \frac{C}{r^2}\Big(
    \int_{Q_{2r}}|D\ue|^2dxdt 
    +\int_{Q_{2r}}(|D\ue|^2+\epsilon)^{\frac{2-\gamma}{2}}dxdt\Big) \\
    &\quad
    +\epsilon\Big(\frac{C}{r^2}\int_{Q_{2r}}\big|\ln(|D\ue|^2+\epsilon)\big|dxdt \\
    &\quad
    +C\int_{B_{2r}}\big|\ln\big(|D\ue(x,t_0)|^2+\epsilon\big)\big|dx
    \Big)
    \end{aligned}
\end{equation}
where $C=C(p,\gamma)>0$.
By \cite{imbertjs19}, for any $Q_R\Subset U_{t_1,t_2}$ there exist positive constants $\alpha\in(0,1)$ and $C>0$, that are allowed to depend on $p$, $\gamma$, $\dist(Q_R,\partial U_{t_1,t_2})$ and $\|u\|_{L^\infty(U_{t_1,t_2})}$,
such that
\begin{equation} \label{eq:Boundedness-of-Dues-in-Holderspace}
    \|D\ue\|_{C^{\alpha}(Q_R)}\leq C. 
\end{equation}
Arzel\`a-Ascoli theorem gives that $\ue$ and $D\ue$ both converge locally uniformly, up to a subsequence, and 
$$ \ue\xrightarrow{\epsilon\to0}\bar{u}
\quad\text{and}\quad
D\ue\xrightarrow{\epsilon\to0}D\bar{u} $$
for some continuous function $\bar{u}\colon U_{t_1,t_2}\to\R$, which by a barrier argument is continuous up to the parabolic boundary, and 
whose spatial gradient $D\bar{u}$ is locally continuous.

By the well known \cite{giga06} stability properties of viscosity solutions,  $\bar{u}$ is a viscosity solution to  
\begin{equation}
\begin{cases}
\begin{aligned}
   \bar{u}_t-|D\bar{u}|^\gamma\big(\Delta\bar{u}+(p-2)\ilN\bar{u}\big)=0
   \quad&\text{in } U_{t_1,t_2};\\
   \bar{u}=u
   \quad&\text{on } \partial_pU_{t_1,t_2}.
\end{aligned}
\end{cases}
\end{equation}
By the uniqueness theorem for viscosity solutions \cite{ohnumas97}, we conclude that $\bar{u}=u$.

By employing bound \eqref{eq:Boundedness-of-Dues-in-Holderspace}, we find that the right hand side of \eqref{eq:RegularizedQuantitativeBoundforSecondDerivativesinBall} is bounded from above by a constant independent of $\epsilon$.
Thus $\{D^2\ue\}_{\epsilon}$ is bounded in $L^2(Q_{r})$, and consequently we may extract a subsequence that converges weakly in $L^2(Q_r)$. Further, using integration by parts, we see that the limit is $D^2u$, and thus $D^2u\in L^2_\loc(\Om_T)$. Finally, we conclude that
\begin{align*}
    \int_{Q_{r}}|D^2u|^2dxdt
    &\leq \liminf_{\epsilon\to 0}\int_{Q_{r}}|D^2\ue|^2dxdt \\
    &\leq \liminf_{\epsilon\to 0} \Bigg(\frac{C}{r^2}\Big(\int_{Q_{2r}} |D\ue|^2 dxdt 
    +\int_{Q_{2r}}(|D\ue|^2+\epsilon)^{\frac{2-\gamma}{2}}dxdt\Big) \\
    &\quad
    +\epsilon\Big(\frac{C}{r^2}\int_{Q_{2r}}\big|\ln(|D\ue|^2+\epsilon)\big|dxdt
    +C\int_{B_{2r}}\big|\ln\big(|D\ue(x,t_0)|^2+\epsilon\big)\big|dx
    \Big)\Bigg) \\
    &= \frac{C}{r^2}\Big(\int_{Q_{2r}}|Du|^2 dxdt 
    +\int_{Q_{2r}}|Du|^{2-\gamma} dxdt\Big),
\end{align*}
which is the desired estimate.
\end{proof}

It is possible to improve the ranges in Theorem \ref{resultforwholep}. However, the computations get more technical, even if they follow the same ideas as above, and thus we have chosen to omit them. In any case the question whether the full range obtained in the smooth case in Proposition \ref{prop:smooth} can also be obtained here remains an open problem. 

Next we give the proof of Corollary \ref{cor:TimeDerivative-w22}.

\begin{proof}[Proof of Corollary \ref{cor:TimeDerivative-w22}]
Assume that $\ue$ is a smooth solution to \eqref{eq:PDEwithGeneralDegeneracy}, and observe
\begin{align*}
\abs{\ue_t}=& \abs{ (|D\ue|^2+\epsilon)^{\gamma/2}\Big(\Delta\ue+(p-2)\frac{\il\ue}{|D\ue|^2+\epsilon}\Big)}\nonumber\\
\leq& (|D\ue|^2+\epsilon)^{\gamma/2}(\abs{\Delta\ue}+\abs{p-2}\abs{D^2\ue}   )\nonumber\\
\leq & (p+2)(|D\ue|^2+\epsilon)^{\gamma/2}\abs{D^2\ue}.
\end{align*}
As above, the spatial gradient is H\"{o}lder continuous
and since $\gamma$ is nonnegative, we have in $Q_{2r}$
$$(|D\ue|^2+\epsilon)^{\gamma/2}\leq C.$$
For all $ Q_r\subset Q_{2r}\Subset \Om _T$,
we have
\begin{align*}
&\int_{Q_r}  |\ue_t|^2  dxdt \nonumber \\
     \leq&(p+2)^2\int_{Q_r} (|D\ue|^2+\epsilon)^{\gamma} |D^2\ue|^2  dxdt \nonumber \\
     \leq& (p+2)^2 \norm{(|D\ue|^2+\epsilon)^{\gamma}}_{L^\infty(Q_r)} \int_{Q_r}    |D^2\ue|^2  dxdt. r
\end{align*}

Then we use (\ref{eq:RegularizedQuantitativeBoundforSecondDerivativesinBall}) estimate the right hand side of the above estimate.
Similarly to the proof of Theorem \ref{resultforwholep}, up to a subsequence, $\{ \ue_t  \}_{\ez}$ converges weakly in $L^2(Q_r)$.
By integration by parts, the weak limit is $u_t$. In particular $u_t$ exists as a function and $u_t\in L^2_{\loc}(\Om_T)$.
\end{proof}
%--------------------------------------------------------------------
\appendix
%--------------------------------------------------------------------
\section{Proof of Lemma \ref{lem:lemma-for-final-estimate}}
Next we prove Lemma \ref{lem:lemma-for-final-estimate}. 
For convenience of the reader, we recall its statement here: 
Let $S$ be as in (\ref{eq:S}),  and $u\colon\Omega_T\to\R$ be a smooth solution to \eqref{eq:RoadmapRegularized}.
(If $\epsilon=0$, we assume in addition that $Du\neq0$.)
Suppose that we can find weights $w_1,w_2,w_3,w_4\in\R$ such that
\begin{equation} \label{eq:lower-bound-for-S}
    S\geq \lambda(|Du|^2+\epsilon)^{\frac{p-2+s}{2}}|D^2u|^2 \quad\text{a.e.\ in space in }\Omega_T,
\end{equation}
for some constant $\lambda=\lambda(n,p,\gamma,s,w_1,w_2,w_3,w_4)>0$.
If $s\neq \gamma-p$, then for any concentric parabolic cylinders $Q_r\subset Q_{2r}\Subset\Omega_T$ with center point $(x_0,t_0)\in\Om_T$ we have the estimate
\begin{equation} \label{eq:lemma-for-final-estimate-A}
    \begin{aligned}
    \int_{Q_r}&\abs{D\big((|Du|^2+\epsilon)^{\frac{p-2+s}{4}}Du\big)}^2dxdt\\
    \leq&
    \frac{C}{r^2}\Big(
    \int_{Q_{2r}}(|Du|^2+\epsilon)^{\frac{p-2+s}{2}}|Du|^2dxdt
    +\int_{Q_{2r}}(|Du|^2+\epsilon)^{\frac{p+s-\gamma}{2}}dxdt\Big) \\
    &+\epsilon\Big(\frac{C}{r^2}\int_{Q_{2r}}\abs{\ln(|Du|^2+\epsilon)}dxdt 
    +C\int_{B_{2r}}\abs{\ln\big(|Du(x,t_0)|^2+\epsilon\big)}dx
    \Big)
    \end{aligned}
\end{equation}
where $C=C(n,p,\gamma,s,\lambda,w_1,w_2,w_3,w_4)>0$.
%--------------------------------------------------------------------
\begin{proof}[Proof of Lemma \ref{lem:lemma-for-final-estimate}]
Let us assume that $s\neq \gamma-p$ and $s\neq \gamma-p+2$. As remarked after the lemma when $s\neq \gamma-p+2$, the logarithmic term does not appear in $S$ and in the estimate of the lemma.
Assumption \eqref{eq:lower-bound-for-S} can be written as
\begin{equation} \label{eq:key-inequality}
\begin{aligned}
    \lambda(|Du|^2+&\epsilon)^{\frac{p-2+s}{2}}|D^2u|^2\\
    &\leq 
    w_1\diverg\big((|Du|^2+\epsilon)^{\frac{p-2+s}{2}}(D^2uDu-\Delta uDu)\big) \\
    & \quad
    +w_2
    \diverg\big(u_t(|Du|^2+\epsilon)^{\frac{p-2+s-\gamma}{2}} Du\big)  
    -w_2\Big(\frac{(|Du|^2+\epsilon)^{\frac{p+s-\gamma}{2}}}{p+s-\gamma}\Big)_t \\
    & \quad
    +\epsilon w_3\diverg\big((|Du|^2+\epsilon)^{\frac{p-4+s}{2}}(D^2uDu-\Delta uDu)\big) \\
    &\quad
    +\epsilon w_4
    \diverg\big(u_t(|Du|^2+\epsilon)^{\frac{p-4+s-\gamma}{2}} Du\big)  
    -\epsilon w_4\Big(\frac{(|Du|^2+\epsilon)^{\frac{p-2+s-\gamma}{2}}}{p-2+s-\gamma}\Big)_t.
\end{aligned}
\end{equation}
Let us fix any concentric parabolic cylinders $Q_r\subset Q_{2r}\Subset \Om_T$ and then select a nonnegative cutoff function 
$\phi\colon\Rn\times[0,t_0]\to[0,1]$ such that
\begin{equation} \label{eq:cutoff}
    \phi\equiv1\enskip\text{in }Q_r,
    \quad
    \phi\equiv0\enskip\text{outside }Q_{2r},
    \quad
    |D\phi|\leq \frac{C}{r}
    \quad\text{and}\quad
    |\phi_t|\leq \frac{C}{r^2}
\end{equation}
for some absolute constant $C>0$.
We multiply \eqref{eq:key-inequality} with $\phi^2$ and integrate over $Q_{2r}$,
apply integration by parts to each integral on the right hand side to obtain
\begin{equation} \label{eq:Estimate-without-log}
\begin{aligned}
    \lambda\int_{Q_{2r}}(|Du|^2+&\epsilon)^{\frac{p-2+s}{2}}|D^2u|^2\phi^2dxdt \\
    &\leq
    -2w_1\int_{Q_{2r}}(|Du|^2+\epsilon)^{\frac{p-2+s}{2}}\la D^2uDu-\Delta uDu,D\phi\ra \phi dxdt \\
    &\quad
    -2w_2\int_{Q_{2r}} u_t(|Du|^2+\epsilon)^{\frac{p-2+s-\gamma}{2}} \la Du,D\phi\ra\phi dxdt \\
    &\quad
    +\frac{2w_2}{p+s-\gamma}\int_{Q_{2r}}(|Du|^2+\epsilon)^{\frac{p+s-\gamma}{2}}\phi_t\phi dxdt \\
    &\quad
    -2\epsilon w_3\int_{Q_{2r}}(|Du|^2+\epsilon)^{\frac{p-4+s}{2}}\la D^2uDu-\Delta uDu,D\phi\ra \phi dxdt \\
    &\quad
    -2\epsilon w_4
    \int_{Q_{2r}}u_t(|Du|^2+\epsilon)^{\frac{p-4+s-\gamma}{2}}\la Du,D\phi\ra\phi dxdt \\
    &\quad
    +\frac{2\epsilon w_4}{p-2+s-\gamma}\int_{Q_{2r}}(|Du|^2+\epsilon)^{\frac{p-2+s-\gamma}{2}}\phi_t\phi dxdt.
\end{aligned}
\end{equation}
Above we dropped the nonpositive boundary terms that appear when we integrate by parts with respect to time.
Next we take absolute values and estimate $\eps/(|Du|^2+\epsilon)\le 1$
in the last three  integrals of the above display. We arrive at
\begin{equation}
\begin{aligned}
    &\lambda\int_{Q_{2r}}(|Du|^2+\epsilon)^{\frac{p-2+s}{2}}|D^2u|^2\phi^2dxdt \\
    &\leq
    C\Big(\int_{Q_{2r}}(|Du|^2+\epsilon)^{\frac{p-2+s}{2}}|D^2u||Du||D\phi|\phi dxdt \\
    &\quad
    +\int_{Q_{2r}} |u_t|(|Du|^2+\epsilon)^{\frac{p-2+s-\gamma}{2}}|Du||D\phi|\phi dxdt \\
    &\quad
    +\int_{Q_{2r}}(|Du|^2+\epsilon)^{\frac{p+s-\gamma}{2}}|\phi_t|\phi dxdt\Big),
\end{aligned}
\end{equation}
where $C=C(n,p,\gamma,s,w_1,w_2,w_3,w_4)>0$. 
By Young's inequality
\begin{equation}
\begin{aligned}
    &(\lambda-2\eta)\int_{Q_{2r}}(|Du|^2+\epsilon)^{\frac{p-2+s}{2}}|D^2u|^2\phi^2dxdt \\
    &\leq
    \frac{C}{\eta}\int_{Q_{2r}}(|Du|^2+\epsilon)^{\frac{p-2+s}{2}}|Du|^2|D\phi|^2dxdt
    +C\int_{Q_{2r}}(|Du|^2+\epsilon)^{\frac{p+s-\gamma}{2}}|\phi_t|\phi dxdt,
\end{aligned}
\end{equation}
for any $\eta>0$ and some $C=C(n,p,\gamma,s,w_1,w_2,w_3,w_4)>0$.
Above we also employed equation \eqref{eq:PDE} and estimated
\begin{equation} \label{eq:Estimate-for-time-derivative}
    (|Du|^2+\epsilon)^{-\gamma/2}|u_t|
    =\abs{\Delta u+(p-2)\frac{\il u}{|Du|^2+\epsilon}}\leq C|D^2u| 
\end{equation}
for some $C=C(n,p)>0$. 
Finally, we select $\eta>0$ small enough and employ \eqref{eq:cutoff} together with the fact that 
\begin{equation} \label{eq:derivative-of-vectorfield}
    \abs{D\big((|Du|^2+\epsilon)^{\frac{p-2+s}{4}}Du\big)}^2\le C(|Du|^2+\epsilon)^{\frac{p-2+s}{2}}|D^2u|^2
\end{equation}
where $C=C(p,s)>0$, to arrive to the desired estimate.

Now, let us assume that $s=\gamma-p+2$.
Assumption \eqref{eq:lower-bound-for-S} is now
\begin{align} \label{eq:key-inequality-log}
    \lambda(|Du|^2+\epsilon)^{\gamma/2}|D^2u|^2
    &\leq
    w_1\diverg\big((|Du|^2+\epsilon)^{\gamma/2}(D^2uDu-\Delta uDu)\big)  \nonumber \\
    &\quad
    +w_2
    \diverg(u_tDu)
    -\frac12 w_2 ( |Du|^2+\epsilon )_t \\
    &\quad
    +\epsilon w_3\diverg\big((|Du|^2+\epsilon)^{\frac{\gamma-2}{2}}(D^2uDu-\Delta uDu)\big)  \nonumber\\
    &\quad
    +\epsilon w_4
    \diverg\big(u_t(|Du|^2+\epsilon)^{-1}Du\big) 
    -\frac{\epsilon}2  w_4 \big( \ln (|Du|^2+\epsilon) \big)_t. \nonumber
\end{align}
Let us fix any concentric parabolic cylinders $Q_r\subset Q_{2r}\Subset \Om_T$ and then select a nonnegative cutoff function $\phi\colon\Rn\times[0,t_0]\to[0,1]$ such that \eqref{eq:cutoff} holds.
We multiply \eqref{eq:key-inequality-log} with $\phi^2$, integrate over $Q_{2r}$,
apply integration by parts to each integral on the right hand side to obtain 
\begin{align}
\label{eq:Estimate-with-log}
    &\lambda\int_{Q_{2r}}(|Du|^2+\epsilon)^{\gamma/2}|D^2u|^2\phi^2dxdt \nonumber \\
    &\leq
    -2w_1\int_{Q_{2r}}(|Du|^2+\epsilon)^{\gamma/2}\la D^2uDu-\Delta uDu,D\phi\ra \phi dxdt 
   \nonumber \\
    &\quad -2w_2\int_{Q_{2r}} u_t\la Du,D\phi\ra\phi dxdt
    +w_2\int_{Q_{2r}}(|Du|^2+\epsilon)\phi_t\phi dxdt \nonumber\\
    &\quad
    -2\epsilon w_3\int_{Q_{2r}}(|Du|^2+\epsilon)^{\frac{\gamma-2}{2}}\la D^2uDu-\Delta uDu,D\phi\ra \phi dxdt \\
    &\quad
    -2\epsilon w_4\int_{Q_{2r}}u_t(|Du|^2+\epsilon)^{-1}\la Du,D\phi\ra\phi dxdt
    +\epsilon w_4\int_{Q_{2r}}\ln (|Du|^2+\epsilon)\phi_t\phi dxdt\nonumber\\
    &\quad
    -\frac{\epsilon w_4}{2}\int_{B_{2r}}\ln(|Du(x,t_0)|^2+\epsilon)\phi^2(x,t_0)dx
\end{align}
Above we dropped the nonpositive boundary term that appears when we integrate by parts with respect to time. However, we cannot drop the boundary term that appears from integrating by parts the last term of the right hand side of \eqref{eq:key-inequality-log}, because logarithm may change sign.

Next we take absolute values and employ again the estimate $\eps/(|Du|^2+\epsilon)\le 1$ to arrive at
\begin{equation}
\begin{aligned}
    &\lambda\int_{Q_{2r}}(|Du|^2+\epsilon)^{\gamma/2}|D^2u|^2\phi^2dxdt \\
    &\leq
    C\Big(\int_{Q_{2r}}(|Du|^2+\epsilon)^{\gamma/2}|D^2u||Du||D\phi|\phi dxdt 
    +\int_{Q_{2r}} |u_t||Du||D\phi|\phi dxdt \\
    &\quad
    +\int_{Q_{2r}}(|Du|^2+\epsilon)|\phi_t|\phi dxdt
    +\epsilon\int_{Q_{2r}}\big|\ln(|Du|^2+\epsilon)\big||\phi_t|\phi dxdt \\
    &\quad
    +\epsilon\int_{B_{2r}}\abs{\ln\big(|Du(x,t_0)|^2+\epsilon\big)}\phi^2(x,t_0)dx
    \Big),
\end{aligned}
\end{equation}
where $C=C(n,p,\gamma,w_1,w_2,w_3,w_4)>0$.
By Young's inequality
\begin{equation}
\begin{aligned}
    &(\lambda-2\eta)\int_{Q_{2r}}(|Du|^2+\epsilon)^{\gamma/2}|D^2u|^2\phi^2dxdt \\
    &\leq
    \frac{C}{\eta}\int_{Q_{2r}}(|Du|^2+\epsilon)^{\gamma/2}|Du|^2|D\phi|^2dxdt
    +C\Big(\int_{Q_{2r}}(|Du|^2+\epsilon)|\phi_t|\phi dxdt \\
    &\quad
    +\epsilon\int_{Q_{2r}}\big|\ln(|Du|^2+\epsilon)\big||\phi_t|\phi dxdt
    +\epsilon\int_{B_{2r}}\big|\ln\big(|Du(x,t_0)|^2+\epsilon\big)\big|\phi^2(x,t_0)dx
    \Big)
\end{aligned}
\end{equation}
for any $\eta>0$ and some $C=C(n,p,\gamma,w_1,w_2,w_3,w_4)>0$.
Above we also employed estimate \eqref{eq:Estimate-for-time-derivative}.
Finally, we select $\eta>0$ small enough and employ \eqref{eq:cutoff} and \eqref{eq:derivative-of-vectorfield} to arrive to the desired estimate.
\end{proof}
%--------------------------------------------------------------------
\section{Positive definiteness condition for the coefficient matrix}

In the proof of Lemma~\ref{lem:smooth-key-lemma}, we wrote one of the key estimates as
\begin{equation*}
\begin{aligned}
    &|Du|^{p-2+s}
    \Big\{
    w_1(p+s)|D_T|Du||^2
    +Q
    \Big\}\le S
\end{aligned}
\end{equation*}
where
\begin{align*}
    Q
    &=
    \Big(w_2-\frac{n-2}{n-1}w_1\Big)(\Delta_T u)^2
    +w_2(p-1)(p-1+s-\gamma)(\ilN u)^2 \\
    &\quad
    +\big(w_2(2p-2+s-\gamma)-w_1(p+s)\big)\Delta_Tu\ilN u.
\end{align*}
This can also be written as
$$ Q=\la\bar{x},M\bar{x}\ra, $$
where $\bar{x}=(\Delta_T u,\ilN u)^T\in \R^2$ is a vector and
$$ M=
\begin{bmatrix}
w_2- \displaystyle{\frac{n-2}{n-1}} w_1 & \displaystyle{\frac{1}{2}}\big(w_2(2p-2+s-\gamma)-w_1(p+s)\big) \\
\displaystyle{\frac{1}{2}}\big(w_2(2p-2+s-\gamma)-w_1(p+s)\big) & w_2(p-1)(p-1+s-\gamma)
\end{bmatrix}.
$$
Then we stated that if $w_1=1$ and the range condition is satisfied, we can select $w_2=w_2(n,p,\gamma,s)>0$ in such a way that $Q$ is positive definite, which then allows us to get rid of the excess terms. Next we prove this fact.
%--------------------------------------------------------------------
\begin{lemma} \label{lem:Range-of-w2}
Let $n\geq 2$, $1<p<\infty$, $-1<\gamma<\infty$, $w_1=1$ and let M be as above.
Then  if
$$ s>\max\Big\{-1-\frac{p-1}{n-1},\gamma+1-p\Big\}, $$
there is $w_2=w_2(n,p,\gamma,s)>0$  such that $M$ is uniformly bounded positive definite (with a uniform constant). 
\end{lemma}
%--------------------------------------------------------------------
\begin{proof} 
We will show that $\operatorname{det}(M)>0$ and $w_2-\frac{n-2}{n-1}>0$ with uniform lower bound, and thus by Sylvester's condition $M$ is uniformly bounded positive definite with a uniform
constant.
We fix $w_1=1$ and introduce the following shorthand notation,
\begin{equation} \label{eq:Shorthand-notation-1}
    P:=p-1\quad\text{and}\quad K:=\gamma+1,
\end{equation}
and 
\begin{equation} \label{eq:Shorthand-notation-2}
    G:=p-1+s-\gamma\quad\text{and}\quad E:=s+1+\frac{p-1}{n-1}.
\end{equation}
We observe that $P, K, G, E>0$ under the assumptions of the lemma. 
Using this notation, one has
$$ M=
\begin{bmatrix}
w_2- \displaystyle{\frac{n-2}{n-1}}   & \displaystyle{\frac{1}{2}}\big(w_2(P+G)- (K+G)\big) \\
\displaystyle{\frac{1}{2}}\big(w_2(P+G)- (K+G)\big) & w_2P\cdot G
\end{bmatrix}
.$$
Then we rewrite the determinant
\begin{align*}
    \operatorname{det}(M)
    =a\Big(w_2-\frac{n-2}{n-1}\Big)^2+b\Big(w_2-\frac{n-2}{n-1}\Big)+c
\end{align*}
where
\begin{align}
   a=-\frac{1}{4}(G-P)^2,\quad  b=P\cdot E+\frac{1}{2}(G-P)\Big(\frac{G}{n-1}+K-\frac{(n-2)P}{n-1}\Big) \nonumber
\end{align}
and
$$ c=-\frac{1}{4}\Big(\frac{G}{n-1}+K-\frac{(n-2)P}{n-1}\Big)^2. $$
The discriminant of such a polynomial is
\begin{align*}
    b^2-4ac=G\cdot P\cdot E\Big(\frac{G}{n-1}+K\Big).
\end{align*}
Notice that $b^2-4ac>0$ and hence our polynomial has two distinct roots, unless $G=P$, in which case our polynomial is of the first order and has one root. Moreover $\operatorname{det}(M)>0$ if and only if $w_2-\frac{n-2}{n-1}$ lies between these roots, that is,
\begin{equation} \label{eq:Between-roots}
    Root_+<w_2-\frac{n-2}{n-1}<Root_-, 
\end{equation}
where
\begin{align*}
    Root_{\pm}
    &=\frac{-\Big(P\cdot E+\frac{1}{2}(G-P)\big(\frac{G}{n-1}+K-\frac{(n-2)P}{n-1}\big)\Big)\pm\sqrt{b^2-4ac}}{-\frac{1}{2}(G-P)^2} \\
    &=\frac{\Big(\sqrt{P\cdot E}\mp\sqrt{G(\frac{G}{n-1}+K)}\Big)^2}{(G-P)^2}\\
    &=\Bigg(\frac{\sqrt{E}}{\sqrt{G}\pm\sqrt{P}}-\frac{\sqrt{G}\Big(\sqrt{E}-\sqrt{\frac{G}{n-1}+K}\Big)}
    {(\sqrt{G}+\sqrt{P})(\sqrt{G}-\sqrt{P})}\Bigg)^2\geq0.
\end{align*}
These formulas are valid if $G\neq P$. Indeed, recall that $a< 0$ if $G\neq P$, then 
\begin{align*}
    Root_{-}-Root_+
    &= -\frac{\sqrt{b^2-4ac}}{a}\\
    &= 4\frac{\sqrt{G\cdot P\cdot E(\frac{G}{n-1}+K )}}{(G-P)^2}>0.
\end{align*}

On the other hand, by
$$ \lim_{G\to P}E=\lim_{s\to\gamma}\Big(s+1+\frac{p-1}{n-1}\Big)=\frac{P}{n-1}+K, $$
and by l'Hopital's rule, one has
\begin{align*}
   \lim_{G\to P}\frac{\sqrt{E}-\sqrt{\frac{G}{n-1}+K}}
    {\sqrt{G}-\sqrt{P}}
    =\frac{(n-2)\sqrt{P}}{(n-1)\sqrt{\frac{P}{n-1}+K}}.
\end{align*}
We conclude that for the smaller root
\begin{align} \label{eq:Root+}
    Root_+\xrightarrow{G\to P}
    \Bigg(\frac{\sqrt{\frac{P}{n-1}+K}}{2\sqrt{P}}-\frac{(n-2)\sqrt{P}}{2(n-1)\sqrt{\frac{P}{n-1}+K}}\Bigg)^2.
\end{align}
For the bigger root, it is easy to see that
\begin{equation} \label{eq:Root-}
    Root_-\xrightarrow{G\to P}\infty.
\end{equation}
The proof is finished.
\end{proof}

%--------------------------------------------------------------------
\section*{Acknowledgement}
The first author was supported by China Scholarship Council, no. 202006020186. The third author was supported  by  the Academy of Finland, Center of Excellence in Randomness and Structures and the Academy of Finland, project 308759.
%--------------------------------------------------------------------

\def\cprime{$'$} \def\cprime{$'$} \def\cprime{$'$}

\tiny{Data availability statement required by the journal: Data sharing not applicable to this article as no datasets were generated or analysed during the current study.}
%--------------------------------------------------------------------
\end{document}